\documentclass{amsart}
\usepackage{amssymb,amscd}
\usepackage[all]{xy}

\usepackage{eufrak}

\newcommand{\co}{\mathfrak{o}}
\newcommand{\crr}{\mathfrak{r}}

\renewcommand{\P}{\mathbb P}

\newcommand{\Z}{\mathbb Z}
\newcommand{\Q}{\mathbb Q}
\newcommand{\R}{\mathbb R}
\newcommand{\C}{\mathbb C}
\newcommand{\N}{\mathbb N}

\newcommand{\cG}{\mathcal G}

\newcommand{\cM}{\mathcal M}

\newcommand{\cO}{\mathcal O}

\newcommand{\cV}{\mathcal V}

\renewcommand{\t}{\widetilde}

\newcommand{\tX}{\widetilde X}

\newcommand{\cha}{\check \alpha}

\newcommand{\m}{\mathfrak m}

\DeclareMathOperator{\Hom}{Hom}

\DeclareMathOperator{\ord}{ord}

\DeclareMathOperator{\lcm}{lcm}

\DeclareMathOperator{\Sf}{Sf}

\renewcommand{\:}{\colon}

\newcommand{\ub}[1]{\underline{#1}}
\newcommand{\bk}{\ub k}
\newcommand{\ba}{\ub {\alpha}}
\newcommand{\bo}{\ub {\omega}}

\newcommand{\PI}{P_{\overline{I}_X}}

\newcommand{\fl}[1]{\left\lfloor #1 \right\rfloor}
\newcommand{\ce}[1]{\left\lceil #1 \right\rceil}
\newcommand{\gi}[1]{\langle #1 \rangle}

\newcommand{\defset}[2]{{\left\{#1\,\left| \,#2 \right. \right\}}}
\newcommand{\X}{(X,o)}

\newcommand{\ten}{\circle*{0.3}}

\numberwithin{equation}{section}
\numberwithin{equation}{subsection}

\theoremstyle{plain}
\newtheorem{theorem}[equation]{Theorem}
\newtheorem{lemma}[equation]{Lemma}
\newtheorem{proposition}[equation]{Proposition}
\newtheorem{corollary}[equation]{Corollary}

\newtheorem{thm}[equation]{Theorem}
\newtheorem{cor}[equation]{Corollary}
\newtheorem{lem}[equation]{Lemma}
\newtheorem{prop}[equation]{Proposition}

\theoremstyle{definition}
\newtheorem{example}[equation]{Example}
\newtheorem{remark}[equation]{Remark}

\newtheorem{definition}[equation]{Definition}

\newtheorem{defn}[equation]{Definition}

\newtheorem{ex}[equation]{Example}

\newtheorem{rem}[equation]{Remark}

\numberwithin{equation}{section}
\numberwithin{equation}{subsection}

\begin{document}
\title[The embedding dimension]{The embedding dimension of weighted homogeneous
surface singularities}

\author{Andr\'as N\'emethi}
\address{R\'enyi Institute of Mathematics, Budapest, Hungary.}
\email{nemethi@renyi.hu}

\author{Tomohiro Okuma}
\address{Department of Education, Yamagata University,
 Yamagata 990-8560, Japan.}
\email{okuma@e.yamagata-u.ac.jp}

\subjclass[2000]{Primary 14B05, 32S25;
Secondary 14J17}

\keywords{Surface singularities, weighted homogeneous singularities, Poincar\'e series, embedding dimension,
Seifert invariants, rational singularities, automorphic forms}

\thanks{The first author is partially supported by
Marie Curie grant and OTKA
grants; the second author by KAKENHI 20540060. }

\begin{abstract}
We analyze the embedding dimension  of a normal weighted homogeneous surface singularity, and
more generally, the Poincar\'e series of the minimal set of generators of the graded algebra of regular functions,
provided that the link of the germs is a rational homology sphere.
In the case of several sub-families we provide explicit formulas in terms of the Seifert invariants
(generalizing results of Wagreich and VanDyke), and we also provide key examples showing that, in general,
these invariants are not topological. We extend the discussion to the case of splice--quotient singularities
with star--shaped graph as well.
\end{abstract}

\maketitle

\section{Introduction}

Let $(X,o)$ be a normal weighted homogeneous surface singularity defined over the
complex number field $\C$. Then $(X,o)$ is the germ at the origin of an affine variety
$X$ with a good $\C^*$--action. Our goal is to determine the minimal set of generators
of the graded algebra $G_X=\sum_{l\geq 0}(G_X)_l$ of regular functions on $X$.
This numerically is codified  in the Poincar\'e series
$$P_{\m_X/\m_X^2}(t)=\sum_{l\geq 0}\dim\,(\m_X/\m_X^2)_l\ t^l,$$
where $\m_X$ is the homogeneous maximal ideal of $G_X$, and
$\dim (\m_X/\m_X^2)_l$ is exactly the number of generators of degree $l$
(in a minimal  set of generators). Notice that
$e.d.(X,o):=P_{\m_X/\m_X^2}(1)$ is the embedding dimension of $(X,o)$.

We will assume that the link of $(X,o)$ is a rational homology sphere. This means that the graph of
the minimal good resolution (or, equivalently, the minimal plumbing graph of the link) is star--shaped,
and all the irreducible exceptional divisors are rational. In particular, the link is a 3--dimensional
Seifert manifold (of genus zero). Usually, it is characterized by its Seifert invariants
$(b_0,(\alpha_i,\omega_i)_{i=1}^\nu)$. Here $-b_0$ is the
self--intersection number of the central curve, $\nu$ is the number of legs.

In fact, our effort to understand $P_{\m_X/\m_X^2}$ is  part of a rather intense activity
which targets the topological characterization of several analytic invariants of weighted homogeneous
(or more generally, the splice--quotient) singularities. For example, if the link is a rational homology sphere, then the following invariants  can be recovered from the link (i.e  from the Seifert invariants):
the Poincar\'e series of $G_X$ and the geometric genus by \cite{pinkham}, the equisingularity type of the universal abelian cover \cite{neumann.abel}, the multiplicity of $(X,o)$ \cite{n.sq}.
In fact, by \cite{n.sq}, the Poincar\'e series of the multi--variable filtration of $G_X$ associated with the valuations of the irreducible components of the minimal good resolution is also determined topologically (proving the so--called Campillo--Delgado--Gusein-Zade identity). This basically says  that a numerical invariant is topological, provided that it can be expressed as dimension of a vector space identified by divisors supported on  the exceptional set.

On the other hand, it was known (at least by specialists) that $P_{\m_X/\m_X^2}$,
in general, is not topological; in particular it has no divisorial description
(a fact, which makes its  computation even more subtle).

Our goal goes beyond finding or analyzing examples showing the non--topological behaviour
of some Poincar\'e series (or embedding dimensions) based on deformation of some equations. Our aim is  
 to develop a strategy
based on which we are able to decide if for the analytic structures supported on a given fixed  
topological type the coefficients of $P_{\m_X/\m_X^2}$ might jump or not. 

The main message of the article is that  using only 
topological/combinatorial arguments, one can always
 decide whether $P_{\m_X/\m_X^2}$ is constant or not 
on a given fixed topological type. 
Moreover, we try to find the boundaries of those cases/families when 
$P_{\m_X/\m_X^2}$ is topological, respectively is not. 

In order to do this, 
first we search for families when $P_{\m_X/\m_X^2}(t)$ is topological
and we try to push the positive results as close to the 
`boundaries' as possible. The  (topological) restrictions we found 
(and  which guarantee the topological nature
of $P_{\m_X/\m_X^2}(t)$) are grouped in several  classes.\\

\noindent {\bf Theorem.} {\it In the following cases $P_{\m_X/\m_X^2}(t)$ is topological:

\vspace{2mm}

(A) \ The order $\co$ of the generic $S^1$--fibre in the link is `small',
e.g. $\co=1$ (see Theorem \ref{c:toped}), or  $\co\leq \min_i\{\alpha/\tilde{\alpha_i}\}$
(see Proposition \ref{prop:osmall}). 

\vspace{2mm}

(B) The number of legs is $\nu\leq 5$ (cf. Theorem \ref{cor:nusmall}). 

\vspace{2mm}

(C) \ $b_0-\nu$ is positive, or negative with small absolute value.
This includes the minimal rational case ($b_0\geq \nu$) generalizing results of 
Wagreich and VanDyke \cite{vd,wa,wa2}, see Theorem \ref{th:LP}; and the situation 
when $G_X$ is the graded ring of automorphic forms relative to a Fuchsian group of 
first kind (completing the list of Wagreich \cite{wa2}), cf. Theorem \ref{th:W}. 

\vspace{2mm}

In cases (A) and (C) we provide explicit formulae.
We also extend (by semicontinuity argument) the
`classical' results of Artin, Laufer and the first author that for rational,
or Gorenstein elliptic singularities, not only $\m_X/\m_X^2$ (which has a divisorial description 
via the fundamental cycle) but $P_{\m_X/\m_X^2}(t)$ too is topological.}

\vspace{2mm}

  On the other hand, we provide key examples (sitting at the `boundary' of the above positive results) when
  $P_{\m_X/\m_X^2}(t)$ depends on the analytic parameters, even computing the `discriminants' when the
  embedding dimension jumps. In the last section we extend the discussion to the case of splice--quotients
  (associated with   star--shaped graphs)  as well.\\
  
{\bf Acknowledgements.} The second  author thanks the BudAlgGeo Marie
Curie Fellowship for Transfer of Knowledge for supporting his
visit at the R\'enyi Institute of Mathematics, Budapest, Hungary,
where the most part of this  paper was done. He is also grateful to
the members of the R\'enyi Institute for the warm hospitality.

In some of the computations we used the SINGULAR program \cite{SING}. 

We also thank the referees for their valuable comments which improved the text.

\section{Preliminaries}
In this section we introduce our notations and review Neumann's
theorem on the universal abelian covers of weighted
homogeneous surface singularities with $\Q$-homology sphere
links.

Let $\X$ be a normal surface singularity with a
$\Q$-homology sphere link $\Sigma$. Let  $\pi\:\tX \to
X$ be the minimal good resolution with the exceptional divisor $E$.
Then $E$ is a tree of rational curves.
Let $\{E_v\}_{v \in \cV}$ denote the set of irreducible components
of $E$ and $-b_v$ the self-intersection number
$E_v^2$.
The set $\cV$ is regarded as the set of vertices of the dual
graph $\Gamma$ of $E$ (i.e., the resolution graph associated with $\pi$).
Let $L$ denote the group of divisors supported on $E$, namely
$
L=\sum _{v\in \cV}\Z E_v
$.
We call an element of $L$ (resp. $L\otimes \Q$) a cycle
(resp. $\Q$-cycle).
Since the intersection matrix $I(E):=(E_v\cdot E_w)$ is negative
definite,
 for each $v \in \cV$ there exists an effective $\Q$-cycle
$E^*_v$ such that $E^*_v\cdot E_w=-\delta_{vw}$ for every $w
\in \cV$,  where
$\delta _{vw}$ denotes the Kronecker delta.
Set
$
L^*=\sum_{v \in \cV}\Z E_v^*
$.
Note that the finite group $H:=L^*/L$ is naturally isomorphic
to $H_1(\Sigma, \Z)$. 

\subsection{The Seifert invariants}\label{SEI}

Suppose that $\Gamma$ is a star-shaped graph with central
vertex $0 \in \cV$ and it has $\nu\ge 3$ legs.
It is well-known that $\Gamma$ is determined by the
so-called Seifert invariants
$\{(\alpha_i,\omega_i)\}_{i=1}^{\nu}$ and the orbifold Euler number $e$ defined as follows.
If the $i$-th leg has the form
\begin{center}
\setlength{\unitlength}{0.6cm}
 \begin{picture}(12,2)(0,-1)
\put(0,0){\ten}
\put(2,0){\ten}
\put(0,0){\line(1,0){3}}
\put(3.5,0){$\ldots$}
\put(5,0){\line(1,0){1}}
\put(6,0){\ten}
\put(-0.5,-1){$-b_0$}
\put(1.5,-1){$-b_{i1}$}
\put(5.5,-1){$-b_{ir_{i}}$}
 \end{picture}
\end{center}
(where the far left--vertex corresponds to the `central curve'
$E_0$) then  the positive integers $\alpha_i$ and
$\omega_i$ are defined by the (negative)
continued fraction expansion :
$$
\frac{\alpha_i}{\omega_i}=\mathrm{cf}\,[b_{i1}, \dots, b_{ir_i}]:=b_{i1}-
\cfrac{1}{b_{i2}-\cfrac{1}{\ddots-\cfrac{1}{b_{ir_i}}}} , \quad
 \gcd(\alpha_i,\omega_i)=1, \ 0<\omega_i<\alpha_i.
$$
Moreover, we define $e:=-b_0+\sum_i
\omega_i/\alpha_i$. Note that $e<0$ if and only if the
intersection matrix $I(E)$ is negative definite.
By \cite[(11.1)]{nem.o-s}
\begin{equation}\label{eq:E.E}
\begin{array}{ll}
E^*_0\cdot E^*_0&= e^{-1};\\
E^*_0\cdot E^*_{i r_i}&=(e\alpha_i)^{-1} \ \mbox{for $1\leq i\leq \nu$};\\
E^*_{ir_i}\cdot E^*_{jr_j}&=\left\{\begin{array}{ll}
(e\alpha_i\alpha_j)^{-1} & \mbox{if $i\not= j$, $1\leq i,\, j\leq \nu$};\\
(e\alpha_i^2)^{-1}-\omega_i'/\alpha_i & \mbox{if $i=j$, $1\leq i,\, j\leq \nu$},\end{array}\right.\end{array}
\end{equation}
where $0<\omega_i'<\alpha_i$ and $\omega_i\omega_i'\equiv 1$ modulo $\alpha_i$.

If $\X$ is a  weighted homogeneous surface singularity (i.e. a
surface singularity with a good $\C^*$-action), then
$\Gamma$ is automatically star-shaped and the complex structure is
completely recovered from the Seifert invariants and the
configuration of the points $\{P_i:=E_0\cap E_{i1}\}_{i=1}^{\nu} \subset
E_0$.
In fact, cf. \cite{pinkham},
  the graded affine coordinate ring $G_X=\bigoplus_{l\ge
0}(G_X)_l$ is given by
\begin{equation}\label{eq:DIV}
(G_X)_l=H^0(E_0, \cO_{E_0}(D^{(l)})), \quad \mbox{with}\quad
D^{(l)}=l(-E_0|_{E_0})-\sum_i\ce{l\omega_i/\alpha_i}P_i,
\end{equation}
where for $r\in \R$, $\ce r$ denotes the smallest integer
greater than or equal to $r$.

In the sequel we write $\alpha:=\lcm \{\alpha_1, \dots ,\alpha_{\nu}\}$ and
the following notation (with abbreviation $\ub 1=(1,
\dots, 1)\in \Z^n$) for the Seifert invariants:
$$
\Sf:=(b_0, \ba, \bo), \quad \ba:=(\alpha_1, \dots ,
\alpha_{\nu}), \quad \bo:=(\omega_1, \dots , \omega_{\nu}).
$$

\subsection{Neumann's theorem and some consequences}\label{ss:neumann}

Assume that $\X$ is a weighted homogeneous surface
singularity and  $\nu\ge 3$. Then there exists a {\itshape universal abelian cover}
$(X^{ab},o) \to \X$ of normal surface singularities that induces an unramified
Galois covering $X^{ab}\setminus\{o\} \to X\setminus \{o\}$ with
Galois group $H_1(\Sigma,\Z)$; if $\Sigma'$ denotes the link
of $X^{ab}$,  then the natural covering $\Sigma' \to \Sigma$
is the universal abelian cover in the topological sense.

\begin{thm}\label{th:N}{\cite{neumann.abel}}
The universal abelian cover $X^{ab}$ is a Brieskorn
 complete intersection
singularity
 $$
\defset{(z_i)\in \C^{\nu}}{f_j:=a_{j1}z_1^{\alpha_1}+\cdots a_{j\nu}z_{\nu}^{\alpha_{\nu}}=0, \quad j=1,
\dots , \nu-2},
$$
where  every maximal minor of  the matrix $(a_{ji})$ does
 not vanish (i.e. $(a_{ji})$
has `full rank').
\end{thm}

In fact, by  row operations, the matrix $(a_{ji})$ can be transformed into
\begin{equation}\label{eq:mat}
  \begin{pmatrix}
 p_1 & q_1& 1 & 0 & \cdots & 0  \\
 p_2 & q_2 & 0 & 1 & \cdots & 0  \\
 \vdots & \vdots & \vdots & \vdots & \ddots  \\
 p_{\nu-2} & q_{\nu-2} & 0 & 0 & \cdots & 1
 \end{pmatrix}
\end{equation}
such that all the entries   $p_i$ and $q_i$ are nonzero, and
 $p_iq_j-p_jq_i \ne 0$ for any $i\ne j$;
moreover, $[p_i:q_i]$ are the projective coordinates
of the points $\{P_3, \dots,  P_{\nu}\} \subset \P^1$, while the remaining two points are
$P_{1}=[1:0]$ and $P_2=[0:1]$.

We define the weights $\mathrm{wt}(z_i)=(|e|\alpha_i)^{-1} \in \Q$.
This induces a grading of the polynomial ring  $R=\C[z_1, \dots , z_{\nu}]$ by
$z^{\bk}:=z_1^{k_1}\cdots z_\nu^{k_\nu}\in R_k$ if and only
if $\sum k_iE^*_{ir_i}\cdot E^*_0=-k$. 
Since all the {\itshape splice  polynomials} $f_j$ are weighted homogeneous of
degree $|e|^{-1}$, there is a naturally
induced grading on the affine coordinate ring of $X^{ab}$ too.


The group $H=L^*/L$ is generated by the classes of $\{E_{ir_i}^*\}_{i=1}^{\nu}$ and  acts
on the polynomial ring $R$ (and/or on $\C^{\nu}$)
as follows. The class  $[E_{ir_i}^*]$ acts
by the $\nu \times \nu$ diagonal
matrix $[e_{i1}, \dots , e_{i\nu}]$,
where $e_{ij}=\exp(2\pi\sqrt{-1} E_{ir_i}^*\cdot
E^*_{jr_j})$, $i=1, \dots, \nu$, cf.   \cite[\S 5]{nw-CIuac}.
By this action $X=X^{ab}/H$;
the invariant subring of $R$ is denoted by $R^H$.

Let $\hat H=\Hom (H, \C^*)$.
For any character $\lambda \in \hat H$, the
$\lambda $--eigenspace $R^{\lambda}$ is 
$$
\defset{f \in R}{h\cdot f=\lambda(h)f \text{ for all }
h \in H}.
$$

A computation shows that for the character
$\mu \in \hat H$ defined by $\mu([E^*_{ir_i}])=
\exp(2\pi\sqrt{-1}E_0^*\cdot  E_{ir_i}^*)$ one has
 $\{f_1, \dots, f_{\nu-2}\} \subset
R^{\mu}$.
We have the following facts:
\begin{lem}\label{lem:OO}
 Let $I_X \subset R^H$  be the ideal generated by
 $R^{\mu^{-1}}\cdot\{f_1, \dots, f_{\nu-2}\}$.
Then the affine coordinate  ring $G_X$ of $X$ is isomorphic  to $R^H/I_X$.

Set $\co:=|e|\alpha$, which is the order of $[E^*_0]$ in $L^*/L$
(cf.  \cite{neumann.abel}). Then
$\co=1$  if and only if the splice functions $f_j$ are in $R^H$
(i.e. $\mu$ is the trivial character).
\end{lem}


\section{Properties of the graded ring $G_X$}

\subsection{The Hilbert/Poincar\'e  series of $G_X$.}
 For any graded vector space $V=\oplus _{l\geq 0}V_l$ we define (as usual) its
 Poincar\'e series
 $P_V(t):=\sum_{l\geq 0} \dim V_l\,t^l.$
 Since $E_0=\P^1$, Pinkham's result (\ref{eq:DIV}) reads as
 \begin{equation}\label{eq:PI}
 P_{G_X}(t)=\sum_{l\geq 0}\max(0,s_l+1)\, t^l, \end{equation}
 where (for any $l\geq 0$) we set
 \begin{equation}\label{eq:SL}
  s_l :=\deg D^{(l)}=lb_0-\sum_i \ce{l\omega_i/\alpha_i} \in \Z.
 \end{equation}
 Although $P_{G_X}$ contains considerably less information than the graded ring $G_X$ itself,
 still it determines rather strong numerical analytic invariants.
 (E.g., it determines the geometric genus
 $p_g$ of $X$, or the log discrepancy  of $E_0$ as well, facts which are less obvious and  which will be clear from the next discussion.)

 Definitely, the series
 $$P(t):=\sum_{l\geq 0} \chi (\cO_{E_0}(D^{(l)}))t^l=\sum_{l\geq 0}(s_l+1)t^l
$$
is more `regular' (e.g., it is polynomial periodic), and it determines both $P_{G_X}$ and
$$P_{H^1}(t):=\sum_{l\geq 0}\dim H^1(E_0,\cO_{E_0}(D^{(l)}))t^l=
\sum_{l\geq 0}\max(0,-s_l-1)\, t^l.$$
By definition, $P=P_{G_X}-P_{H^1}$.
Since $e<0$, $\lim_{l\to\infty}s_l=\infty$, and
$P_{H^1}(t)$ is a polynomial (and by \cite{pinkham},  $p_g=P_{H^1}(1)$).
In fact,
$$-(\alpha-1)|e|-\nu\leq s_l-\ce{l/\alpha}\alpha|e|\leq -1.$$
Indeed, if one writes $l$ as $\ce{l/\alpha}\alpha-d$, and
$\t s_d=db_0-\sum \fl{d\omega_i/\alpha_i}$, then $s_l=\ce{l/\alpha}\alpha|e|-\t s_d$,
while $\t s_d\ge 1$ by {\cite[11.5]{nem.o-s}} and $\t s_d\leq (\alpha-1)|e|+\nu$ by a computation.

\begin{rem} From $P_{G_X}(t)$ one can recover $P(t)$, hence $P_{H^1}(t)$ too.
Indeed, one can show (see e.g. \cite[Corollary 1.5]{stan.comb}) that $P(t)$ can be written
as a rational function $p(t)/q(t)$ with $p, q \in \C[t]$ and  $\deg p<\deg q$.
Then, if one writes (in a unique way)
$P_{G_X}(t)$ as $p(t)/q(t)+r(t)$ with $p,\, q, \, r\in\C[t]$ and $\deg p<\deg q$, then $p/q=P$ and $r=P_{H^1}$.
\end{rem}
The next subsection provides a (topological) upper bound for the degree of the polynomial $P_{H^1}$.

\subsection{The degree of $P_{H^1}$.}
Let us recall Pinkham's construction of the graded ring $G_X$
 (see \S 3 and \S 5 of \cite{pinkham}).
There exists a finite Galois cover $\rho\: E'\to E_0$ with
 Galois group $G$ such that $\rho|_{E'\setminus \rho^{-1}(\{P_1,
 \dots, P_{\nu}\})}$ is unramified and the ramification
 index of any point of $\rho^{-1}(P_i)$ is $\alpha_i$.
Let $D$ denote  the  (rational) Pinkham-Demazure divisor on
 $E_0$, i.e., $D=-E_0|_{E_0}-\sum_i
 \frac{\omega_i}{\alpha_i}P_i$. Clearly $\deg D=-e$ and
$D^{(l)}=\fl{lD}$, the integral part of $lD$, for every
 nonnegative integer  $l$.
Then $D':=\rho^*D$ is an integral divisor and invariant
 under the action of $G$.
Thus $G$ acts on the spaces $H^j(E',\cO_{E'}(lD'))$ with
 $j=0,1$. The invariant subspace is denoted by  $H^j(E',\cO_{E'}(lD'))^G$.
Pinkham \cite[\S 5]{pinkham} proved the following:
\begin{equation}
 \label{eq:inv}
H^j(E_0,\cO_{E_0}(D^{(l)}))\cong H^j(E',\cO_{E'}(lD'))^G,
\quad j=0,1.
\end{equation}

Since $X$ is $\Q$-Gorenstein, it follows from the argument
of \cite[\S 1]{Dol-link} that there exists a rational number
$\gamma$ such that the canonical divisor $K_{E'}$ on $E'$ is
$\Q$-linearly equivalent to $\gamma D'$, and that $\gamma$
is an integer if $X$ is Gorenstein.
By the Hurwitz formula, we have
$$
 K_{E'}=\rho^*K_{E_0}+\sum
_{i=1}^{\nu}\sum _{Q\in \rho^{-1}(P_i)}(\alpha_i-1)Q.
$$
Since $\#\rho^{-1}(P_i)=|G|/\alpha_i$ for every $i$, we
obtain
$$
\gamma=\deg K_{E'}/\deg D'
=\frac{1}{|e|}\left(\nu-2-\sum_{i=1}^{\nu}\frac{1}{\alpha_i}\right).
$$

\begin{prop}\label{prop:Gamma}
For $l>\gamma$, we have $H^1(\cO_{E_0}(D^{(l)}))=0$, i.e. $s_l\geq -1$. Hence
$\deg P_{H^1}(t)\leq  \max (0,\gamma)$.
\end{prop}

The following proposition shows that the bound in (\ref{prop:Gamma}) is sharp:

\begin{prop}\label{prop:sharp} Assume that $(X,o)$ is Gorenstein, but not of type A-D-E (i.e. it is not rational).
Then the degree of $P_{H^1}(t)$ is exactly $\gamma$.  Moreover, the
coefficient of $t^\gamma$ is 1.
In particular, if $(X,o)$ is minimally elliptic, then
$P_{H^1}(t)=t^\gamma$.
\end{prop}

\begin{proof}
 [Proof of (\ref{prop:Gamma}) and (\ref{prop:sharp})]
If $l>\gamma$, then $H^1(\cO_{E'}(lD'))=0$ since $\deg
 (K_{E'}-lD')<0$.
Thus (\ref{prop:Gamma}) follows from \eqref{eq:inv}.
Assume that $(X,o)$ is Gorenstein, but not rational.
Then $\gamma$ is a nonnegative integer (cf. \cite[5.8]{pinkham}) and
$$
H^1(\cO_{E_0}(D^{(\gamma)}))\cong
 H^1(\cO_{E'}(K_{E'}))^G\cong \C.
$$
Hence (\ref{prop:sharp}) follows too.
\end{proof}

\begin{rem}
\begin{enumerate}
 \item $s_{\alpha}=\alpha(b_0-\sum \omega_i/\alpha_i)=\alpha |e|=\co$ is a positive integer.
 \item  For any $l\geq 0$, $s_{l+\alpha}=s_l+s_\alpha=s_l+\co>s_l$.
\end{enumerate}
\end{rem}

\begin{cor}\label{c:sd0}
Assume $l>\alpha+\gamma$.  Then $s_l\ge 0$, hence  $(G_X)_l$ is non--trivial.
\end{cor}

\begin{remark}{\bf (Different interpretations of $\mathbf{\gamma}$)} \

(a) Recall that any Cohen--Macaulay (positively) graded $\C$--algebra $S$ admits the so--called
$a$--invariant $a(S)\in \Z$, for details see the article of  Goto and Watanabe
\cite[(3.1.4)]{G-W}, or \cite[(3.6.13)]{c-m}.
Let $G_{X^{ab}}$ denote the affine coordinate ring of $X^{ab}$.
Since $X^{ab}$ is a complete intersection, cf. (\ref{th:N}),
its  $a$-invariant $a(G_{X^{ab}})$ is determined by
\cite[(3.6.14-15)]{c-m}. This, in terms of Seifert invariants,
reads as follows (see also \cite[\S 3]{o.pg-splice}):
$$
a(G_{X^{ab}})=(\nu-2)\alpha-\sum_{i=1}^{\nu}\frac{\alpha}{\alpha_i}
=\co\gamma,
$$
where $\co=|e|\alpha$ is the order of $[E^*_0]$.

(b) 
Let $-Z_K$ be the canonical cycle associated with the canonical line bundle of $\widetilde{X}$, i.e. $Z_K\in L^*$ satisfies the adjunction relations $Z_K\cdot E_v=E_v^2+2$ for all
$v\in\cV$. Then (see e.g. \cite[(11.1)]{nem.o-s}) the coefficient of $E_0$ in $Z_K$ is exactly
$1+\gamma$. Hence $-\gamma$ is the log discrepancy of $E_0$
(sometimes $\gamma$ is called also `the exponent of $(X,o)$').

(c)
The rational number $\gamma$ can also be recovered already from the asymptotic  behaviour of the coefficients of
$P(t)$ or $P_{G_X}(t)$. Indeed, by a result of Dolgachev,
the Laurent expansions of both $P(t)$ and $ P_{G_X}(t)$ at $t=1$ have the form  (cf. \cite[(4.7)]{neumann.abel}, or \cite[Proposition 3.4]{swII} for one more term of the expansion):
$$|e|\cdot \Big(\ \frac{1}{(t-1)^2}+\frac{1+\gamma/2}{t-1}+\mbox{regular part}\ \Big).$$
\end{remark}

\subsection{The homogeneous parts $(G_X)_l$ reinterpreted.}\label{s:hpart}
Assume that $\X$ is weighted homogeneous with graded ring $\oplus_{l\geq 0}(G_X)_l$ and graph $\Gamma$
as above.

Let $\cM=\{\,z^{\bk} \, | \, k_i\in \Z_{\ge
0}, \ i=1, \dots , \nu\} \subset R=\C[z_1,\dots ,z_{\nu}]$ be the set of all
monomials.
For any $\bk=(k_1, \dots, k_{\nu})$, set $d_{\bk}:=(\sum_ik_i/\alpha_i)/|e| \in \Q$,
the degree of $z^{\bk}$.
We define a homomorphism of semigroups $\varphi\:(\Z_{\ge 0})^{\nu}\to
\Q^{\nu}$ by
$$
\varphi(\bk)=(l_1, \dots, l_{\nu}), \ \mbox{where} \  l_i:=(k_i+d_{\bk}\omega_i)/\alpha_i.
$$
Let $\cM^H$ denote the set of invariant monomials (with respect to the action of $H$).

\begin{lem}\label{INV}
$z^{\bk} \in \cM^H$ if and only if
 $\varphi(\bk)\in (\Z_{\ge 0})^{\nu}$ and $d_{\bk} \in
 \Z_{\ge 0}$.
\end{lem}
\begin{proof}
 We first recall the description of $H$ from \cite[\S 1]{neumann.abel}
 (see also \cite[\S 2.5]{swII}).
Let $h_i$ denote the class $[E_{ir_i}^*]$ ($i=1,
\dots, \nu$) and $h_0$ the class $[E_0^*]$ in $H$.
Then
\begin{equation}
 \label{eq:H}
H=\gi{h_0, h_1, \dots ,h_{\nu}\; \vert\; h_0^{b_0}\prod
_{i=1}^{\nu}h_i^{-\omega_i}=1, \; h_0^{-1}h_i^{\alpha_i}=1\;
(i=1, \dots , \nu) }.
\end{equation} Moreover, all the relations  among $\{h_i\}_{i=0}^\nu$ have the form
$$
h_0^{lb_0-\sum l_i}\prod h_i^{-l\omega_i+l_i\alpha_i}=1, \quad \mbox{for some } \
 l, l_1, \dots ,l_{\nu} \in \Z.
$$
Next, notice that  $z^{\bk} \in \cM^H$ if and only if
$\prod_{i=1}^\nu h_i^{k_i}=1$. If this is happening, then
 there exist $l, l_1, \dots ,l_{\nu} \in \Z$ such that
 $lb_0-\sum l_i=0$ and $-l\omega_i+l_i\alpha_i=k_i$ for $i=1,
 \dots ,\nu$.
Then we have $l_i=(k_i+l\omega_i)/\alpha_i$ and
$$
\sum k_i/\alpha_i=\sum (l_i-l\omega_i/\alpha_i)=lb_0-l\sum\omega_i/\alpha_i=l|e|.
$$
Hence $l=d_{\bk}$ and $\varphi(\bk)=(l_i)$.
The converse is now easy.
\end{proof}

\begin{cor} For any $l\geq 0$,
the linear space  $(R^H)_l \subset R^H$ of forms of degree
 $l$ is spanned  by
$\defset{z^{\bk}}{\sum_ik_i/\alpha_i=l |e|,\; k_i+l\omega_i
 \equiv 0 \pmod{\alpha_i},\ i=1, \dots, \nu}$.
\end{cor}

\begin{definition}
For each $1\le i \le \nu$ and $l\geq 0$ define:
\begin{itemize}
 \item $\beta_i:=\alpha_i-\omega_i$;
 \item $a_i:=z_i^{\alpha_i}$;
 \item $M_l:=\prod z_i^{\{l\beta_i/\alpha_i\}\alpha_i}$, where for
       $r \in \R$, $\{r\}$ denotes the fractional part of $r$.
\end{itemize}
\end{definition}
Note  that $M_l=M_{l'}$ if $l\equiv l'\pmod{ \alpha}$.

Consider  $z^{\bk} \in (\cM^H)_l$ and take $(l_i)=\varphi(\bk)$.
Then  $n_i:=l_i-\ce{l\omega_i/\alpha_i}$  is a non--negative integer
 which  satisfy  $\sum n_i=s_l$ and
$k_i-n_i\alpha_i=\{l\beta_i/\alpha_i\}\alpha_i$. Hence:

\begin{prop}\label{p:trans}
 For any $l\geq 0$ one has
$$
(\cM^H)_l=\begin{cases}
	   M_l\cdot \defset{\prod a_i^{n_i}}{\sum n_i=s_l, \; \ub n=(n_i)\in
 (\Z_{\ge 0})^{\nu}} & \text{if $s_l\ge 0$,} \\ \emptyset & \text{if $s_l<0$.}
	  \end{cases}
$$
\end{prop}

\begin{defn}
 Let $A=\C[a_1, \dots ,a_{\nu}]$ be the polynomial ring graded by
$\mathrm {wt}(a_i)=1$. For any $l$,
define the map $\psi_l\: (R^H)_l \to A_{s_l}$ by $\psi(f)=f/M_l$.
\end{defn}

Then Proposition  (\ref{p:trans}) implies that $\psi_l$ is an
isomorphism of $\C$-linear spaces.
By this correspondence, the splice polynomials
$f_j=\sum _ia_{ji}z_i^{\alpha_i}$ (cf. \ref{th:N}) transform into
the linear forms $\ell_j=\sum _ia_{ji}a_i$ of $A$ ($j=1,\ldots,\nu$).
Hence, every element of $\psi_l(I_X)$ has the form $\sum_jq_j\ell_j$, where $q_j$ are arbitrary
$(s_l-1)$--forms of $A$. In particular, if $I$ denotes the ideal generated by the
linear forms $\{\ell_j\}_{j=1}^{\nu-2}$, then
\begin{equation}\label{eq:GL}
(G_X)_l=A_{s_l}/\langle \textstyle{\sum_j} q_j\ell_j\rangle=(A/I)_{s_l}.
\end{equation}
Notice that via  this representation we easily can recover
Pinkham's formula (\ref{eq:PI}). Indeed, if $s_l\geq 0$, then
using the linear forms $\{\ell_j\}$, the variables $a_3, \dots , a_{\nu}$
 can be eliminated; hence
$\dim (G_X) _l=\dim \C[a_1,a_2]_{s_l}=s_l+1$.

\section{The square of the maximal ideal}\label{s:m^2}

\subsection{The general picture}\label{ss:GP}

Let $\m \subset R^H$ denote the homogeneous maximal ideal
 and $\m_X$ the homogeneous maximal ideal of $G_X=R^H/I_X$.
The aim of this section is to  compute the dimension $Q(l)$ of the $\C$-linear space
$$
(\m_X/\m_X^2)_l=(\m/I_X+\m^2)_l.
$$
These numbers can be inserted in the Poincar\'e series of $\m_X/\m_X^2$:
$$P_{\m_X/\m_X^2}(t)=\sum_{l\geq 1}Q(l)t^l.$$
This is a polynomial with $P_{\m_X/\m_X^2}(1)=e.d.(X,o)$, the embedding dimension of $(X,o)$. Indeed,
if one considers a minimal set of homogeneous generators of the $\C$--algebra $G_X$, then $Q(l)$ is exactly the number of generators of degree $l$.

Our method relies on  the description  developed in the subsection (\ref{s:hpart}).

First we describe the structure of $\m^2$ in terms of the monomials of $\cM^H$.

\begin{definition} Let $\N^*$ denote the set of positive integers.
For $l\in \N^*$ we  set
$$
\Lambda_l:=\defset{(l_1,l_2)\in (\N^*)^2}{l_1+l_2=l, s_{l_1}\ge 0, s_{l_2}\ge 0}.
$$
The monomial $z^{\bk} \in (\cM^H)_l$ is called  {\itshape linear} if either
 $\Lambda_l=\emptyset$, or $z^{\bk}\not\in (\cM^H)_{l_1}\cdot
 (\cM^H)_{l_2}$ for any $(l_1,l_2)\in \Lambda_l$.
\end{definition}
The linear monomials form  a set of minimal generators of $\m$.


Next, we transport this structure on the polynomial ring $A$. In order not to make confusions between the
degrees of the monomials from  $R$ and $A$, we emphasize the corresponding degrees by writing  $\deg_R$ and  $\deg_A$  respectively.

\begin{defn}\label{eps}
(a) For $\ub l=(l_1, l_2)\in \Lambda _l$ and $1\leq i\leq \nu$,  set $$\epsilon_{i,\ub
 l}=\{l_1\beta_i/\alpha_i\}+\{l_2\beta_i/\alpha_i\}-\{l\beta_i/\alpha_i\}\in\{0,1\}.$$
(b) $X_l$ will denote the set of (automatically reduced)
monomials of $A$  defined by
$$
X_l:=\defset{M_{l_1}M_{l_2}/M_l}{(l_1,l_2)\in \Lambda_l}
=\defset{\ \prod_{i=1}^{\nu} a_i^{\epsilon_i,\ub l}\ }{\ \ub l\in \Lambda_l}.
$$
(c)  Let $J(l)$ be the ideal generated by  $X_l$ in $A$
 ($J(l)=(0)$ if $X_l=\emptyset$).
\end{defn}
From the definition of the ideals $J(l)$ and the map  $\psi_l$ one has:

\begin{prop}\label{l:m2}
 $\psi_l((\m^2)_l)=J(l)_{s_l}$.
In particular,
 $Q(l)=\dim (A/I+J(l))_{s_l}$.
\end{prop}


For the convenience of the reader, we list some properties of the monomials $M_l$
and of the generators $X_l$ which might be helpful in the concrete computations.

\begin{lem}\label{r:trivial}
 \begin{enumerate}
 \item $\deg_R M_l+s_l/|e|=l$;
 \item \label{2}$ \sum_i \epsilon_{i,\ub l}= \deg_A\left(M_{l_1}M_{l_2}/M_l
	\right)=s_l-s_{l_1}-s_{l_2} \le \min\{\nu, s_l\}$ for any  $\ub l=(l_1, l_2)$;
 \item for any $m,n\in\Z_{\ge 0}$,
if  $(l_1,l_2)\in \Lambda_l$, then
 $(l_1+m\alpha,l_2+n\alpha)	\in
 \Lambda_{l+(m+n)\alpha}$ too, and
 $ \epsilon_{i,(l_1,l_2)}=\epsilon_{i,(l_1+m\alpha,l_2+n\alpha)}$; hence
$X_l\subset X_{l+\alpha}$;
 \item \label{4} if  $s_{l_2}\ge 0$
then  $\epsilon_{i,(\alpha,l_2)}=0$ for every $i$; hence $1\in X_{\alpha+l_2}$;

 \item \label{5} if $\Lambda_l\ne \emptyset$ and $s_l=0$, then $1\in X_l$ (cf. (\ref{2}));
 \item if $l>2\alpha+\gamma$, then $1\in X_{l}$ (cf. (\ref{c:sd0}) and (\ref{5})).
\end{enumerate}
\end{lem}

\subsection{Is $Q(l)$ topological ?} Let us recall/comment the formula  $Q(l)=\dim (A/I+J(l))_{s_l}$
from  (\ref{l:m2}). Here, for any $l\in \N^*$, the ideal $J(l)$ is  combinatorial
depending only on the Seifert invariants
$\{(\alpha_i,\omega_i)\}_i$ of the legs. The integer
$s_l$ is also combinatorial (for it, one also needs the integer $b_0$). On the other hand,
the ideal $I$ is generated by the `generic' linear forms $\{\ell_j=\sum_ia_{ji}a_i\}_{j=1}^{\nu-2}$, where
the `genericity' means that  the matrix $(a_{ji})$ has {\em full rank}.

By a superficial argument, one might conclude that $Q(l)$ is topological, i.e. for all
matrices $(a_{ji})$ of full rank, $Q(l)$ is the same. But,  this is {\em not the case}:
In the space of full rank matrices, there are some sub--varieties along which the value of  $Q(l)$
might jump.

Therefore, in the sequel our investigation bifurcates into two directions:

\vspace{2mm}

\noindent {\bf 1.}  {\em Find equisingular families of weighted homogeneous
 surface singularities in
 which the  embedding dimension is not constant (i.e.
the embedding dimension cannot be determined from
 the Seifert invariants). Analyze the `discriminant' (the `jump--loci'), and write the corresponding
 equations, deformations.}

 \vspace{2mm}

 \noindent {\bf 2.}
 {\em Find topologically identified families of weighted homogeneous  singularities,  which are characterized by special properties of the  Seifert invariants,  for which
 $P_{\m_X/\m_X^2}(t)$, hence the  embedding dimension too, is topological. Then, determine them from the Seifert invariants. }

\vspace{2mm}

The remaining part of the present article deals with these two directions, providing key positive results and examples in both directions, and trying to find the boundary limit between the two categories.
Let us provide as a warm up, some intuitive easy explanation for both directions how they might appear.

\begin{lem}\label{l:0or1} {\bf \em (`Easy cases' when $Q(l)$ is topological)}

 {\bf I.} If $s_l\ge 0$ and $X_l=\emptyset$, then $Q(l)=s_l+1$. 
 If $s_l=0$ and $X_l\not =\emptyset$, then $Q(l)=0$.

 {\bf II.} Assume that there exists an $i\in\{1,\ldots,\nu\}$ such that
 $a_i\in X_l$. Let $(a_i)$ be the ideal in $A$ generated by $a_i$. Then
\begin{enumerate}
 \item if $J(l)\subset (a_i)$, then $Q(l)=1$.
 \item if $J(l)\not\subset (a_i)$, then $Q(l)=0$.
\end{enumerate}
In particular, $Q(l)$ is topological. If  $s_l\geq 1$ then 
$Q(l)$ is the number of variables appearing in all the monomials of $X_l$
whenever $X_l\not=\emptyset$.
\end{lem}
\begin{proof} 
The first assertion of (I) follows from  $Q(l)=\dim
 \C[a_1,a_2]_{s_l}$, 
and the second from (\ref{r:trivial})(5). In (II)
we may assume $i=1$. In case (1), $J(l)=(a_1)$ and
 $A/I+J(l)\cong \C[a_2]$.
For (2), let $\delta=\min\defset{\deg_Am}{m\in
 X_l\setminus(a_1)}$.
Then $A/I+J(l)\cong \C[a_2]/(a_2^{\delta})$.
Since $\delta\le s_l$, cf. (\ref{r:trivial})(\ref{2}),  $(A/I+J(l))_{s_l}=0$.
\end{proof}

\begin{example}\label{NT} {\bf \em (How can  $Q(l)$  be non--topological?)}
Assume that in some situation $\nu=6$, $X_l=\{a_1a_2,a_3a_4,a_5a_6\}$ and $s_l=2$ (cf. (\ref{ss:223377})).
Consider the linear forms $\ell_j=\sum_ia_{ji}a_i$, where the matrix $(a_{ji})$ is from (\ref{eq:mat}).
Since all $q_i$'s are non--zero, we may assume $q_i=1$ for all $i$.
{\em The full rank condition is equivalent with the fact that all $p_i$'s are non--zero and distinct. }

Then, by eliminating the variables $a_3, \ldots,a_\nu$ using the linear forms,  $Q(l)=\dim (\C[a_1,a_2]/J')_2$,
where $J'$ is generated by
$a_1a_2$, $(p_1a_1+a_2)(p_2a_1+a_2)$ and $(p_3a_1+a_2)(p_4a_1+a_2)$. Therefore,
$$
Q(l)=\begin{cases}
			  0 & \mbox{if} \ p_1p_2-p_3p_4\ne0, \\ 1 & \mbox{if} \ p_1p_2-p_3p_4=0.
			 \end{cases}
$$
Hence, along the `non--topological discriminant' $p_1p_2-p_3p_4=0$, the embedding dimension increases.
\end{example}

\subsection{}\label{ss:4.5} Via the next example we show how the general 
procedure presented above runs. In this example $P_{\m_X/\m_X^2}$ will be topological.

We denote the monomial $\prod a_i^{k_i}$ by $(k_1,
\dots, k_{\nu})$.
For example,  $a_1^2=(2,0,\dots,0)$.

\begin{example}\label{ex:2345}
{\bf The case $\Sf=(2,(2,3,4,5),(1,1,1,4))$.}

Suppose that the numbering of $E_i$'s satisfy
$$
\begin{matrix}
 &&&& E_1 & \\ E_4&E_5&E_6&E_7&E_8&E_2 \\ &&&& E_3 &
\end{matrix},
$$
where $E_8$ is the central curve and $E_i$ is the end
corresponding to $(\alpha_i,\omega_i)$ for $i=1,2,3,4$.
For $\sum _{i=1}^8a_iE_i$, we write $\{a_1, \dots, a_8\}$.
The fundamental invariants are listed below:
\begin{enumerate}
 \item $e=-7/60$, $\alpha=60$, $\co=7$, $\gamma=43/7=6\frac{1}{7}$.
 \item $|H|=14$.
 \item  The fundamental cycle is $Z=\{3,2,2,2,3,4,5,6\}$,
       $p_a(Z)=1$.
 \item  The canonical cycle is $Z_K=\left\{\frac{25}{7},\frac{19}{7},\frac{16}{7},\frac{10}{7},\frac{20}{7
   },\frac{30}{7},\frac{40}{7},\frac{50}{7}\right\}$.
\item  The Hilbert series  is
$$
P_{G_X}(t)=1+t^6+t^8+t^{10}+t^{11}+2 t^{12}+t^{14}+2 t^{15}+2 t^{16}+t^{17}+2
   t^{18}+t^{19}+3 t^{20}+$$ $$\hspace*{1cm}2 t^{21}+2 t^{22}+2 t^{23}+3 t^{24}+2
   t^{25}+3 t^{26}+3 t^{27}+3 t^{28}+2 t^{29}+4 t^{30}+O\left(t^{31}\right).
$$

 \item $P_{H^1}(t)=t$ and $p_g=1$.
 \item
 The degrees of $z_1, z_2,z_3,z_4$ are $\frac{1}{7}(30,20,15,12)$.
\end{enumerate}

\begin{lem}\label{l:lealpha} \

(1) For $\alpha-l\ge 0$,  $s_{\alpha-l}\ge s_{\alpha}-s_l-\nu$.


(2) $s_l\ge 6$ for $l\ge 62$, $s_{61}=5$, $s_{60}=7$.

(3) $s_l\ge 0$ for $l\ge 14$.


(4) $X_{l+14}\ni 1$ for $l\ge \alpha=60$.
\end{lem}
\begin{proof} (1) is elementary.
For  (2) write $s_{l+\alpha}=s_l+s_{\alpha}=s_l+7$. Since $P_{H^1}(t)=t$,
 $s_1=-2$ and $s_l\ge -1$ for $l\ge 2$.
For (3), consider the formula for the Hilbert series (see above).
Hence $s_{l}\le 3$ for $l\le 29$.
By (1), $s_{60-l}\ge 7-s_l-4\ge 0$.
On the other hand, from the same formula of $P_{G_X}$, one has  $s_l\ge 0$ for $14\le l \le 30$.
Now, (4) follows from (\ref{r:trivial})(\ref{4}).
\end{proof}

The following is the list of $(l, s_l, X_l)$ with $s_l\ge
0$,
$X_l\not\ni 1$, and $l\le 74$ (cf. (\ref{l:lealpha})), where $a=(1,0,0,0)$,
$b=(0,1,0,0)$, $c=(0,0,1,0)$, $d=(0,0,0,1)$:
$$
\begin{array}{ccc}
 6 & 0 & \emptyset \\
 8 & 0 & \emptyset \\
 10 & 0 & \emptyset \\
 11 & 0 & \emptyset \\
 12 & 1 & \{c\} \\
 15 & 1 & \emptyset \\
 16 & 1 & \{d,c\} \\
 20 & 2 & \{d,c+d,b+c\} \\
 30 & 3 & \{d,b,b+d,a,a+b+d\} \\
 36 & 4 & \{d,c,c+d,b,b+d,b+c,b+c+d,\\
  \ & \ & \ \               a+c,a+b+c,a+b+c+d\} \\
 60 & 7 & \{d,c,c+d,b,b+d,b+c,b+c+d,\\
  \ & \ & \ \ a+c,a+c+d,a+b+c,a+b+c+d\}
\end{array}
$$
Using Lemma (\ref{l:0or1}) we verify
$$
Q(12)=1, \quad Q(15)=2, \quad Q(j)=0 \text{ for } j=16, 20, 30, 36, 60.
$$
Therefore,  the Hilbert series is
$$P_{\m_X/\m_X^2}(t)=t^6+t^8+t^{10}+t^{11}+t^{12}+2t^{15}.$$
One can check that in this case the difference $P_{\m/\m^2}(t)-P_{\m_X/\m_X^2}$
is not concentrated only in one degree (compare with the results of the next section). 
In order to see this, notice that
 the action of the group $H$ on the variables $(z_1,\ldots,z_4)$ is given by the
 following diagonal matrices
$
[\zeta^{13},\zeta^4,\zeta^3,\zeta^8],
$
where $\zeta$ denotes a primitive 14-th root of unity.
The invariant subring of $\C[z_1, \dots , z_4]$ is
generated by the following 21 monomials (we used  \cite{SING}):
$$
\begin{array}{*{21}{c}}
z_{3}^{2}z_{4} & z_{2}z_{4}^{3} & z_{1}^{2}z_{4}^{2} &
 z_{1}z_{2}z_{3}z_{4} & z_{2}^{2}z_{3}^{2} & z_{1}^{3}z_{3}
 & z_{2}^{3}z_{4}^{2} \\
 z_{1}^{2}z_{2}^{2}z_{4} &
 z_{1}z_{2}^{3}z_{3} & z_{1}^{4}z_{2} &  z_{2}^{5}z_{4}
& z_{1}z_{3}^{5} & z_{1}^{2}z_{2}^{4} & z_{4}^{7} \\
 z_{1}z_{3}z_{4}^{5} & z_{1}^{2}z_{2}z_{3}^{4} & z_{2}^{7} & z_{1}^{8}z_{4} & z_{2}z_{3}^{8} & z_{3}^{14} & z_{1}^{14}
\end{array}
$$
\end{example}

\section{Restrictions regarding $ \co=|e|\alpha$}

In this section we treat our first families when the Poincar\'e polynomial 
(in particular  $e.d.(X,o)$ too) is topological.

\subsection{The case $\co=1$.}\label{ss:O1} Assume that $\co=1$, i.e. $E_0^*\in L$.
In this case the splice functions belong to $R^H$, cf. (\ref{lem:OO}), and their degree is exactly
$\alpha$. We will proceed in several steps. First, we consider the exact sequence
\begin{equation}\label{eq:ES}
0\to \frac{I_X}{I_X\cap \m^2}
\to \frac{\m}{\m^2} \to \frac{\m_X}{\m_X^2} \to 0.
\end{equation}
This is compatible with the weight--decomposition. Denote by $\PI(t)=\sum_{l\geq 1}i_lt^l$  the
Poincar\'e polynomial of $\overline{I}_X:=I_X/(I_X\cap \m^2)$. Then
$$P_{\m_X/\m_X^2}(t)=P_{\m/\m^2}(t)-\PI(t).$$
Clearly, $P_{\m/\m^2}(t)$ is topological, its $l$-th coefficient is $\dim (A/J(l))_{s_l}$, i.e. it is the Poincar\'e polynomial of the linear monomials of $\m$. (In particular, $P_{\m/\m^2}(1)$ is the embedding dimension of the
quotient singularity $\C^\nu/H$.)

The degree $l=\alpha$ (of the splice equations) is of special interest. We claim that the set of invariant monomials
of degree $\alpha$ is $\{z_1^{\alpha_1},\ldots,z_{\nu}^{\alpha_\nu}\}$. Indeed,
  by (\ref{p:trans}), for $l=\alpha$ one has $M_l=1$  and $s_l=\co=1$; hence the invariant monomials of degree $\alpha$ correspond to the  monomials of $A_1$.

In the next Proposition, $h_i=[E_{ir_i}^*]\in H$ as in  (\ref{s:hpart}).
\begin{prop}\label{pro:LIN} \

(1) Set $\cha_i=\lcm (\{\alpha_1, \dots , \alpha_{\nu}\}\setminus \{\alpha_i\})$.
Then  $\ord(h_i)=\alpha_i \cha_i/\alpha $.

(2) $z_i^{\alpha_i}$ is linear
$\Leftrightarrow$  $\ord (h_i)=\alpha_i$ $\Leftrightarrow$ $\alpha=\cha_i$.
\end{prop}
\begin{proof} From  (\ref{eq:H}),
 $H=\gi{h_1, \dots , h_{\nu}\;  \vert \; \prod_{j}h_j^{\omega_j}=1,
 \; h_j^{\alpha_j}=1 \ (1\leq j\leq \nu)}$.
Since $(\alpha_j, \omega_j)=1$, $h_j^{\omega_j}$ generates
 $\gi {h_j}\cong \Z_{\alpha_j}$.
Hence $H\cong \prod_{j}\Z_{\alpha_j}/(1, \dots , 1)$.
Hence (1) follows from the exact sequence
$$
1 \to \gi{h_i} \to \prod_{j}\Z_{\alpha_j}/(1, \dots , 1) \to
\prod_{j\ne i}\Z_{\alpha_j}/(1, \dots , 1) \to 1.
$$
For (2), let us first determine those integers $k_i$ for which $z_i^{k_i}$ is invariant.
By (\ref{INV}), this is equivalent to the facts: $\ord(h_i)|k_i$ and
$\frac{k_i}{\alpha_i}(1+\frac{\alpha}{\alpha_i}\omega_i)\in\Z$. This implies (2).
\end{proof}
Now we are able to  analyze  $\PI$.

\begin{theorem}\label{c:toped} 
Assume that $\co=1$. Then the following facts hold:

(a) $P_{\m_X/\m_X^2}(t)=\sum Q(l)t^l$ is $P_{\m/\m^2}(t)-i_\alpha t^\alpha$, where $i_{\alpha}=\dim I_X/I_X\cap \m^2$.
Hence it is determined completely by $P_{\m/\m^2}(t)$ and $Q(\alpha)$.

(b) $Q(\alpha)=\max(0,2-\# X_{\alpha})$.
In particular, $P_{\m_X/\m_X^2}(t)$ is a topological invariant.
\end{theorem}
\begin{proof}  Since $\{f_j\}\m \subset \m^2$, the linear space $\sum _j\C
 f_j$ generates $I_X/I_X\cap \m^2$. In particular,
$\PI(t)=i_{\alpha}t^{\alpha}$.
This shows (a). 

For (b) note  that $s_{\alpha}=\co=1$ and $M_\alpha=1$. Moreover,
$\alpha$ is the smallest integer $l$ with $M_l=1$, hence $1\not\in X_\alpha$. Hence, by
 (\ref{r:trivial})(\ref{2})  $X_{\alpha} \subset \{a_1, \ldots ,a_{\nu}\}$, and $X_\alpha$ corresponds
 bijectively  to the non--linear monomials of type $z_i^{\alpha_i}$ (which are characterized in
 (\ref{pro:LIN})).
The dimension of the linear space generated by $I$ and
 $X_{\alpha}$ is $r:=\min(\nu, \nu-2+\# X_{\alpha})$.
Hence  $Q(\alpha)=\nu-r$.
\end{proof}

\begin{example}
Consider the case of the $E_6$ singularity.
In this case $H=\Z_3=\defset{\xi\in \C}{\xi^3=1}$, and the universal abelian cover has equations
  $z_1^3+z_2^3+z_3^2=0$. The coordinates $(z_1,z_2,z_3)$ have degrees $(2,2,3)$, $H$ acts on them by
 $(\xi,\overline{\xi},1)$. The linear invariant monomials are $z_1^3, z_2^3, z_1z_2 $ and $z_3$,
 of degree 6,6,4,3. Hence $P_{\m/\m^2}(t)=t^3+t^4+2t^6$. $f_1$ kills a 1--dimensional space of
 degree $\alpha=6$, hence $P_{\m_X/\m_X^2}(t)=t^3+t^4+t^6$. This is compatible with the identity $X_6=\{a_3\}$ (cf. (\ref{c:toped})) as well.
\end{example}

\begin{example} If $\Sf=( 1,(14,21,5),(5,5,2))$ then $\co=1$ and $f_1\in\m^2$. Hence
$P_{\m_X/\m_X^2}=P_{\m/\m^2}$.
\end{example}

\begin{example}\label{LINBIG}
Surprisingly, the degree of $P_{\m/\m^2}$ (hence of $P_{\m_X/\m_X^2}$ too)
can be larger than $\alpha$ (i.e. there may exist linear monomials of $R^H$ of  degree larger than the
degree of the splice equations).

Consider e.g. the graph with Seifert invariants
\begin{gather*}
b_0=1,  \quad \ba=(3,4,5,6,21), \quad \bo=\ub 1.
\end{gather*}
Then $e=-1/420$, $\alpha=420$, $\co=1$.
The degrees of the variables of $R=\C[z_1, \dots ,z_5]$ are
 $140,105,84,70,20$. The group $H$ acts on $R$ via diagonal matrices
$$
[\zeta^4,-1,1,\zeta,\zeta^4], \quad [\zeta^2,1,1,\zeta^4,1],
$$
 where $\zeta$ denotes a primitive 6-th root of unity.
There are 9 linear monomials in $R^H$, namely:
$$
z_5^3, \; z_3\; , \; z_2^2, \; z_2z_4^3, \;  z_1z_2z_4z_5, \; z_1^3, \; z_4^6, \; z_1z_4^4z_5, \;
 z_1^2z_4^2z_5^2.
$$
Their degrees are: 60, 84, 210, 315, 335,  420, 420, 440, 460 respectively.
The two terms of degree 420 are eliminated by the (three) splice equations. Hence
$$P_{\m_X/\m_X^2}(t)=t^{60}+t^{84}+ t^{210}+ t^{315}+t^{ 335}+ t^{ 440}+t^{ 460}.$$
Notice that in this case $\#X_{420}=3$, hence (\ref{c:toped}) gives $Q(420)=0$ as well.
(For further properties of this singularity, see (\ref{LINBIG2}).)
\end{example}

\begin{remark}
Notice that if $H$ is trivial then $e.d.(X,o)=\nu$ by (\ref{th:N}). Furthermore, 
 if $\co=1$ then in (\ref{LINBIG}) we characterize topologically the embedding dimension.
Hence, apparently, the structure of the  group $H$ has subtle influence 
on the topological nature of $e.d.(X,o)$. Nevertheless, at least for the authors,
this connection is still rather hidden and mysterious.

On the other hand, in order to show that the above two cases 
are not accidental and isolated, in the next subsection we will treat one more situation.
\end{remark}

\subsection{The case $\co$ `small'.}\label{ss:OS} We will start with a graphs $\Gamma$ with the property that for each
$i\in\{1,\ldots,\nu\}$, there exists an integer $k_i$, $1\leq k_i\leq \alpha_i$, such that $z_i^{k_i}\in R^H$.

Since $z_i^{k_i}\in R^H$ if and only if $h_i^{k_i}=1$ in $H$,
the above property is equivalent with $\ord(h_i)\leq \alpha_i$ for all $i$.
Since $H/\langle h_i\rangle=\prod_{j\not=i}\Z_{\alpha_j}/(1,\ldots,1)$,  one gets that $\ord(h_i)=
\alpha_i\tilde{\alpha}_i\co/\alpha$. Therefore, $\ord(h_i)\leq \alpha_i$ reads as
\begin{equation}\label{eq:co}
\co\leq \alpha/\tilde{\alpha}_i \ \ \  \ \ \mbox{for all $i$}.
\end{equation}
\begin{proposition}\label{prop:osmall}
Assume now that $\co$ satisfies (\ref{eq:co}), but $\co>1$. Then 
\begin{equation}\label{eq:co2}
P_{\m_X/\m_X^2}(t)\equiv P_{\m/\m^2}(t).
\end{equation}\end{proposition}
This says that the splice equations have absolutely no effect in $P_{\m_X/\m_X^2}(t)$. 

\begin{proof} By (\ref{eq:ES}),
this is equivalent with the inclusion  $I_X\subset \m^2$. Recall that $I_X$ is generated by
expressions of form $f_j\cdot m$, where $m\in R^{\mu^{-1}}$.
Since $\co\not=1$,  $\mu$ is not the trivial character.
Hence $1\not\in R^{\mu^{-1}}$. Let us analyze a product of type
 $z_i^{\alpha_i}m$. It can be rewritten as 
 $z_i^{k_i}\cdot (z_i^{\alpha_i-k_i}\cdot m)$, where the 
 element in the parenthesis is also invariant.  Hence this 
 expression is in $\m^2$, and (\ref{eq:co2}) is proved. \end{proof}

The above case can really appear. Indeed, consider e.g. the Seifert invariants
$\Sf=(1,(3,5,11),(1,1,5))$ (see (\ref{ss:4.5}) for the notation).
Then $\co=|H|=2$ and $\alpha/\tilde{\alpha}_i\in \{3,5,11\}$.
(Clearly,  if the $\alpha_i$'s are pairwise relative prime, and $\co=2$, or more generally,
$\co\leq \min_i(\alpha_i)$, then  the above assumptions are still satisfied.)

On the other hand, merely the identity $\co=2$ is not enough, see e.g. (\ref{ss:223377}).

\section{More combinatorics. The case $\nu\leq 5$.}\label{s:5}

\subsection{A key combinatorial lemma.}\label{ss:key} Recall  that for each $l>0$ the ideal
$J(l)$ is generated by the square--free monomials of $X_l$. In the literature such an ideal is called
{\em Stanley--Reisner ideal}, $A/J(l)$ is a {\em Stanley--Reisner (graded) ring}. Their literature is
very large, see e.g. \cite{MS} and references therein. In particular, their
Hilbert/Poincar\'e series are determined combinatorially and other beautiful combinatorial connections are provided.

Still, in the literature we were not able to find how this ideal behaves with respect to
intersection with a generic hyperplane section, or even more generally, with respect to a
2--codimensional generic  linear section (as it is the case in our situation, cf. the ideal
$I$ in (\ref{s:hpart})). As this article shows, these questions outgrow the {\em combinatorial}
commutative algebra, in general. Nevertheless, the goal of the present section is to establish
combinatorial results about the dimension of some graded parts $(A/J(l)+I))_{s_l}$, at least under some restriction; facts which will guarantee the topological nature of the embedding dimension and of $P_{\m_X/\m_X^2}$.

We keep all the previous  notations; additionally, if $X_l\not=\emptyset$ we write
$n_l:=\#\defset{i}{1\le i \le \nu, \ X_l\subset (a_i)}$,
the number of variables appearing in all the monomials of $X_l$,
and $m_l$ for the number of variables not appearing at all in the set of monomials which form a minimal
set of generators of $J(l)$.
If $n_l>0$ then by reordering the variables we may assume that $J(l)\subset (u)$, where
$u=\prod_{i=\nu-n_l+1}^{\nu}a_i$. Then $J(l) $ has the form
$(u) J'(l)$ for some ideal $J'(l)$ (generated by monomials which do not contain the last $n_l$ variables).

Notice that by (\ref{r:trivial}) the $A$--degree of any monomial of $X_l$ is not larger than $s_l$, hence
$n_l\leq s_l$.

\begin{lem}\label{KEY} \ (i) 
$$\dim \Big(\frac{A}{J(l)+I}\Big)_{s_l}=
\dim \Big(\frac{A}{J'(l)+I}\Big)_{s_l-n_l}+n_l.$$
In particular, $Q(l)\geq n_l$.

\vspace{2mm}

(ii) \ \  If $s_l\ge \nu-m_l-1$, then $Q(l)=n_l$.
\end{lem}
\begin{proof} First we prove (ii)
in several steps.

(a) First step:
If $J\subset A$ is an ideal generated by reduced  monomials
 such that $J\not\subset (a_i)$ for any $i\in \{1, \dots,
 \nu\}$, and if $k\ge \nu-1$, then
 $\left(A/J+I\right)_{k}=0$.

Consider a monomial
$a=\prod_{i=1}^{\nu}a_i^{k_i} \in A_k$ for which we wish to prove that $a\in J+I$.
Set  $y=y(a):=\#\defset{i}{k_i=0}$.
If $y=0$ then $a\in J$ since $J$ is generated by reduced  monomials.
If $y=1$, say $k_j=0$, 
then by the assumption there exists a reduced
 monomial $m\in J$ with $a_j\nmid m$, 
and thus $a\in (m)\subset J$.
Suppose $y=2$ and $k_1=k_2=0$. Since $\sum_{i=3}^{\nu}k_i=k\ge
 \nu-1$,  we may assume that $k_3\ge 2$. Let $a'=a/a_3$.
 Since $a'a_1, \dots, a'a_{\nu}$
are related by $\nu-2$ generic linear equations (induced by $I$), and
 $a'a_1, a'a_2\in J$ by the argument above,  we get that
 $a'a_1, \dots, a'a_{\nu} \in J+I$ as well. By induction on
 $y$, we obtain  that $a\in J+I$ for all $a\in A_k$ and $y(a)$.

(b) Next, start with the same assumptions
(i.e., $J$ is generated by reduced monomials, and there is
no variable $a_i$ appearing in all these monomials), and write
$m$ for the number of variables not appearing at all in the set of monomials which
form a minimal set of generators of $J$.
Then we get $\left(A/J+I\right)_{k}=0$ if $k\geq\nu-m-1$.

Indeed, assume that $a_\nu,\ldots, a_{\nu-m+1}$ are not appearing in the monomials (here we may assume that $m\leq \nu-2$). Assume
that in the equations $l_j=\sum_ia_{ji}a_i$ the matrix $(a_{ji})$ has the form (\ref{eq:mat}); hence by
the last $m$ equations $l_j$ one can eliminate the variables $a_\nu,\ldots, a_{\nu-m+1}$. Doing this, we
find ourself in the situation when we have $\nu-m$ variables, the same $J$ (satisfying the same assumption), hence the statement follows from part (a).

(c) Now we prove the statement of the lemma for arbitrary $n_l$.
Let $\varphi$ denote the natural isomorphism $A/I\to
 A':=\C[a_{1}, a_{2}]$.
Let $J'_l\subset A'$ denote the ideal such that
 $\varphi(J(l))=\varphi(u)J'_l$ (or, take $J'_l:=\varphi(J'(l))$).
Then we have the following exact sequence:
\begin{equation}\label{eq:5.1.1c}
0 \to \left(A'/J'_l\right)_{s_l-n_l}
\xrightarrow{\times \varphi(u)}
\left(A'/\varphi(u)J'_l\right)_{s_l} \to
\left(A'/\varphi(u)A'\right)_{s_l} \to 0.
\end{equation}
$\left(A'/J'_l\right)_{s_l-n_l}$ corresponds to the situation covered by (b)
(with $\nu$ variables, degree $s_l-n_l$ and $n_l+m_l$ of variables not appearing in the generators
of $J'(l)$),
hence  it is zero.
Therefore,  $Q(l)=\dim\left(A'/\varphi(u)A'\right)_{s_l}=n_l$. This proves (ii).
The exact sequence (\ref{eq:5.1.1c}) 
proves part (i) as well, since
$A'/J'_l\approx A/(J'(l)+I)$.
\end{proof}

The `decomposition' (\ref{KEY})(i) can be exploited more:

\begin{theorem}\label{cor:nusmall}
Assume that $\nu-n_l\leq 5$, then $Q(l)$ is topological. Consequently,
$P_{\m_X/\m_X^2}$ is topological provided that $\nu\leq 5$.
\end{theorem}
\begin{proof} By the proof of (\ref{KEY}) we only have to verify that $\dim
(A/(J'(l)+I))_{s_l-n_l}$ is topological. 
 If $\nu-n_l=0$ or $1$, then  $J'(l)=(1)$.
On the other hand, if $a_i\in J'(l)$ for some $1\le i \le \nu$, then $Q(l)$ is
 topological by Lemma (\ref{l:0or1}).
Hence we may assume that $J'(l)$ contains no generator $a_i$ of $A$--degree one.

Assume that $J'(l)$ contains a monomial of degree $2$, say
 $a_1a_2$.
Since $J'(l)$ is not principal, there exists a monomial
 $m\in J'(l)$ such that $(m)\not\subset (a_1a_2)$.
Assume that  $m$ has minimal degree among such monomials.
If $a_1 \mid m$, then  $a_1a_2, a_1^{\deg m}\in J'_l$.
Since the image of other monomials of $J'(l)$ in $J'_l$ is
 of the form $pa_1^{k}+qa_2^k$ modulo $(a_1a_2)$ with $k\ge \deg m$, $Q(l)$ is topological.
After this discussion (up to a permutation of variables)
only the following two cases remain uncovered:
$$
m=a_3a_4,  \ \mbox{or}\ \ m= a_3a_4a_5
$$
For these case, we have
$$
a_1a_2,a_1^2+ca_2^2\in J'_l,  \ \mbox{or}\  \ a_1a_2,a_1^3+ca_2^3\in J'_l, \ c\ne 0,
$$
respectively. Then clearly $Q(l)$ is topological.

Next, we may assume that $J'(l)$ contains no monomial of degree $\le 2$.
Then, automatically  $\nu-n_l\ge 4$.
If  $\nu-n_l=4$, since non of the variables $a_1, \dots,
 a_4$ divide all the monomials of $J'(l)$, $J'(l)$ should contain
all the monomials of degree $3$ in $a_1, \dots,
 a_4$. Then $Q(l)$ is again topological.

Finally, assume that $\nu-n_l=5$. If the minimal degree of monomials of
 $J'(l)$ is $4$, then by a similar argument as above, $Q(l)$ is topological.
Assume that $a_1a_2a_3\in J'(l)$ and
let $c$ denote the number of monomials  of degree $3$ in the
 minimal set $\cG$ of generators of $J'(l)$ which consists of
 monomials.

It is not hard to verify (using again the definition of $J'(l)$)
that $J'(l)$ contains all the monomials of degree
 $4$. Hence, the only unclarified  rank is 3, i.e.,
 $\dim (A'/J'_l)_3$.

If $c=1$, then clearly  $\dim (A'/J'_l)_3$ is
 topological, and
 $\cG$ contains at least $3$ monomials of
 degree $4$. If $c=2$, then we have the  following two possibilities:
$$
a_1a_2a_3, a_3a_4a_5 \in \cG; \ \ a_1a_2a_3, a_2a_3a_4\in \cG.
$$
Then clearly  $\dim (A'/J'_l)_3$ is
 topological, and
 $\cG$ contains at least $1$ (resp. $2$) monomials of
 degree $4$.
The case $c=3$ reduces to:
$$
a_1a_2a_3, a_1a_2a_4, a_3a_4a_5 \in \cG.
$$
The image of $3$ monomials in $J'_l$ is:
$$
\begin{pmatrix}
 0 & p_3 & 1 & 0 \\ 0 & p_4 & 1 & 0 \\ \sigma_3 & \sigma_2 &
 \sigma_1 & 1
\end{pmatrix}
\begin{pmatrix}
 a_1^3 \\ a_1^2a_2 \\ a_1a_2^2 \\a_2^3
\end{pmatrix},
$$
where $a_i=p_ia_1+a_2$ ($i=3,4,5$) and $\sigma_i$ denotes
 the elementary symmetric polynomial of degree $i$ in $p_3,
 p_4, p_5$. We easily see that this matrix has rank $3$,
 which is
 independent of the parameters,  and
 hence $\dim (A'/J'_l)_3$ is topological.

Suppose $c\ge 4$. Then we have the following $6$ cases:
\begin{enumerate}
 \item $a_1a_2a_3, a_1a_2a_4, a_1a_2a_5, a_3a_4a_5 \in \cG,$
 \item $a_1a_2a_3, a_1a_2a_4, a_1a_3a_4, a_2a_3a_4 \in \cG,$
 \item $a_1a_2a_3, a_1a_2a_4, a_1a_3a_4, a_2a_3a_5 \in \cG,$
 \item $a_1a_2a_3, a_1a_2a_4, a_1a_3a_5, a_2a_3a_4 \in \cG,$
 \item $a_1a_2a_3, a_1a_2a_5, a_1a_3a_4, a_2a_3a_4 \in \cG,$
 \item $a_1a_2a_3, a_1a_2a_4, a_1a_3a_5, a_2a_4a_5 \in \cG.$
\end{enumerate}
For (1), the corresponding matrix is the matrix on the left:
\begin{equation}\label{matr} \
\begin{pmatrix}
 0 & p_3 & 1 & 0 \\ 0 & p_4 & 1 & 0  \\ 0 & p_5 & 1 & 0
\\ \sigma_3 & \sigma_2 &
 \sigma_1 & 1
\end{pmatrix}, \hspace{1cm} \
\begin{pmatrix}
 0 & p_3 & 1 & 0 \\ 0 & p_4 & 1 & 0  \\ \sigma_2 & \sigma_1 & 1 & 0
\\ 0 & \sigma_2 & \sigma_1 & 1
\end{pmatrix}.
\end{equation}
Since this matrix has rank $3$, $\dim (A'/J'_l)_3$ is topological.
For case (2), the matrix is the second one in (\ref{matr}). Here
 $\sigma_i$ denotes the elementary symmetric polynomial of degree $i$ in $p_3,  p_4$.
Since this matrix has rank $4$, $\dim (A'/J'_l)_3$ is topological.
In (3)--(6), we have the same type matrices as in the case (2), and their ranks
 are $4$.
\end{proof}


\begin{remark}
In the statement (\ref{cor:nusmall}), the case $\nu=3$ is not surprising,
since in this case the analytic structure has no moduli
(the position of the three points $P_1,P_2,P_3$ in $\P^1$ has no moduli, cf. (\ref{SEI})).
On the other hand, for the cases $\nu=4$ and 5 we do not see  any  explanation other than the
(case by case combinatorial) discussion of the above proof based on the geometry of low--dimensional forms.
\end{remark}

\begin{rem} (Cf. (\ref{NT}).) The above proof really shows the limits of the topological nature of $Q(l)$
from the point of view of the invariants involved in Theorem (\ref{cor:nusmall}). E.g., if
$J'(l)$ is one of the ideals from the following list, then  $Q(l)$ is not topological
(below  $*$ indicates the empty set  or sequence of monomials of
 higher       degrees).
\begin{enumerate}
 \item $s_l=2$,  $J'(l)=(a_1a_2, a_3a_4, a_5a_6, *)$;\\
$s_l=3$,  $J'(l)=(a_1a_2a_3, a_4a_5 a_6, a_7a_8a_9,*)$;\\
 ...
 \item  $s_l=3$,  $J'(l)=(a_1a_2, a_3a_4a_5, a_6a_7a_8, *)$.
\end{enumerate}
\end{rem}

\section{In the `neighborhood' of rational singularities}\label{s:RS}

\subsection{The topological nature of $P_{\m_X/\m_X^2}$
for rational and minimally elliptic germs.}\label{ss:NR}
Recall  that for rational and minimally elliptic singularities,
by results of Artin and Laufer \cite{ar1,lau}, the embedding dimension is topological (see also \cite{n.w}
for the Gorenstein elliptic case). Therefore, it is natural to expect  the topological nature of
$P_{\m_X/\m_X^2}$ for those weighted homogeneous singularities which belong to these families.
The next result shows that this is indeed the case:

\begin{proposition}\label{th:NR}
Assume that the star--shaped graph $\Gamma$ is either rational or
numerically Gorenstein elliptic. Then for any normal
 weighted homogeneous singularity $(X,o)$ with graph
 $\Gamma$,
the polynomial $P_{\m_X/\m_X^2}$ is independent of the
 analytic structure, it depends only on $\Gamma$.
\end{proposition}
\begin{proof}
First note that a weighted homogeneous singularity with
 $Z_K\in L$ is  Gorenstein (hence \cite{n.w} applies).
 By the above description of the coefficients $Q(l)$ of $P_{\m_X/\m_X^2}$, they are upper semi--continuous with respect to the parameter space (i.e. with respect to the full rank matrices
$(a_{ji})$). Since $P_{\m_X/\m_X^2}(1)=e.d.(X,o)$ is independent on the choice of the parameter, all these coefficients have the same property too.
\end{proof}

In the case of  minimal rational and
 automorphic case we provide explicit formulae in terms of Seifert invariants.

\subsection{The minimal rational case.}\label{MinRat} Recall that
$(X,o)$ is minimal rational if $b_0\geq \nu$, i.e. if the  fundamental cycle is reduced.

\begin{lemma}\label{lem:ineq}
  Assume that  $b_0\geq \nu$. Then the following facts hold:

(a) $s_l\geq 0$ for any $l$, hence $X_l\not=\emptyset$ for any $l>0$.\\
 (In particular, the integer $m_l$ is well--defined for any $l>0$, cf. (\ref{ss:key}).)

 (b) $s_l\geq \nu-m_l$ for any $l>0$.
\end{lemma}
\begin{proof}
(a) follows from $l\geq \ce{l\omega_i/\alpha_i}$ ($\dagger$). For (b), notice that
$$s_l-\nu\geq \textstyle{\sum_i}\big(\,l-1-\ce{l\omega_i/\alpha_i}\big).$$ For any fixed $i$,
$l-1-\ce{l\omega_i/\alpha_i}\geq -1$
by ($\dagger$). Moreover, it is $-1$ if and only if $l=\ce{l\omega_i/\alpha_i}$, or $l\beta_i<\alpha_i$.
This implies  $\epsilon_{i,\ub l}=0$ for any $\ub l$, hence $a_i$ is not present in the monomials of
$X_l$.
\end{proof}
(\ref{lem:ineq})(b) combined with the combinatorial lemma (\ref{KEY})(ii) provides:
\begin{corollary}\label{QL} If $\Gamma$ is minimal rational then
 $Q(l)=n_l$ for any $l>0$.
\end{corollary}
 Let us reinterpret the integer $n_l$ in terms of Seifert invariants $\{(\alpha_i,\omega_i)\}_i$.
The next  result connects the arithmetical properties of continued fractions with our construction
(namely with the construction of the monomials from $X_l$).

We consider a pair  $0<\beta<\alpha$, with $\mbox{gcd}(\alpha,\beta)=1$, and the continued fraction expansion
$\alpha/\beta=\mbox{cf}\,[u_1,\ldots,u_\crr] $ (with all $u_i\geq 2$, cf. (\ref{SEI})). For any $1\leq k\leq \crr$, we consider $r_k/t_k:=\mbox{cf}\,[u_1,\ldots,u_k]$, with $\mbox{gcd}(r_k,t_k)=1$ and $r_k>0$.
\begin{proposition}\label{prop:LP}
 Fix $(\alpha,\beta)$ as above and $l\geq 2$.  Then  $l\in \{r_1,\ldots, r_\crr\}$ if and only if
$$\epsilon_j:=\{j\beta/\alpha\}+\{(l-j)\beta/\alpha\}-\{l\beta/\alpha\}=1$$
for all $j\in\{1,\ldots ,l-1\}$.
\end{proposition}
\begin{proof}
First assume that $l>\alpha$. Then $\epsilon_\alpha=0$ and $r_k\leq \alpha$ for all $k$, hence the equivalence
follows. If $l=\alpha$ then $l=r_\crr$ and clearly $\epsilon_j=1$ for all $j$ ($\epsilon_j=0$ can happen only if $\alpha\mid j\beta$ which is impossible). Hence, in the sequel we assume $l<\alpha$.

The proof is based on lattice point count in planar  domains. If $D$ is a closed integral polyhedron in the plane, we denote by  $LP(D^\circ)$ the number of lattice points in its interior.
We denote by $\Delta_{P,Q,R}$ the closed triangle with vertices $P,Q,R$.

Let  $\Delta_l$ be  the
closed triangle determined by $y\leq \beta x/\alpha$,  $x\leq l$ and $y\geq 0$.
Then
$$\sum_{1\leq j\leq l-1}\lfloor j\beta/\alpha\rfloor=LP(\Delta_l^\circ).$$
If we replace $j$ by $l-j$ and we add it to  the previous identity, we get that
$$\sum_{1\leq j\leq l-1}\epsilon_j=(l-1)\lfloor l\beta/\alpha\rfloor -2 LP(\Delta_l^\circ).$$
Hence, $\epsilon_j=1$ for all $j$ if and only if
$2 LP(\Delta_l^\circ)=(l-1)(\lfloor l\beta/\alpha\rfloor-1)$.
Hence, we have to prove that
\begin{equation}\label{eq:LP}
2 LP(\Delta_l^\circ)\geq (l-1)(\lfloor l\beta/\alpha\rfloor-1)
\end{equation}
with equality if and only if $l\in\{r_1,\ldots, r_\crr\}$.

Assume that $l=r_k$ for some $k$. Recall that the convex closure of the  points $\{(r_k,t_k)\}_k$ together with $(\alpha,0$) contains all the lattice points of $\Delta^\circ_{(0,0),(\alpha,0),(\alpha,\beta)}$ ($\dagger$),
see e.g. \cite[(1.6)]{oda}.
Moreover, by Pick theorem, for any triangle $\Delta$, $2LP(\Delta^\circ)=2+2\cdot
Area(\Delta)-\mbox{number of}$ $\mbox{lattice points on the boundary of $\Delta$}$. Therefore, $LP(\Delta_{r_k}^\circ)=LP(\Delta_{(0,0),(r_k,0),(r_k,t_k)})$; this by Pick theorem is
$(r_k-1)(t_k-1)$. Since $l=r_k$ and $\lfloor l\beta/\alpha\rfloor=t_k$ by ($\dagger$), (\ref{eq:LP}) follows.

Next, we show that in (\ref{eq:LP}) one has strict inequality whenever $r_k<l<r_{k+1}$.

Let $\Delta$ be the triangle with vertices $(r_k,t_k),(r_{k+1},t_{k+1})$ and (the non--lattice point)
$(r_{k+1}, r_{k+1}t_k/r_k)$ (with two vertices on the line $y=t_ky/r_k$).
Notice that $(r_{k+1},t_{k+1})$ is above the line $y=t_kx/r_k$, cf. also with (\ref{CF}).
\begin{lemma}\label{lem:LP}
$\Delta^\circ$ contains no lattice points. All the lattice points on the boundary of
$\Delta$, except $(r_{k+1},t_{k+1})$, are sitting on the line $y=t_ky/r_k$.
\end{lemma}\begin{proof}
This basically follows from the following identity of continued fractions:
\begin{equation}\label{CF}
\frac{r_k}{t_k}-\frac{r_{k+1}}{t_{k+1}}=\frac{1}{t_kt_{k+1}}.\end{equation}
Indeed,  considering the slope of the segment with ends
$(x,y)$ and $(r_k,t_k)$, the lattice points in $\Delta$ with $y>t_kx/r_k$ and $x<r_{k+1}$ are
characterized by  $r_k<x<r_{k+1}$ and
$$\frac{t_k}{r_k}
<\frac{y-t_k}{x-r_k}\leq \frac{t_{k+1}-t_k}{r_{k+1}-r_k}.$$
By a computation  and using (\ref{CF}) this transforms into
$$0<yr_k-xt_k\leq \frac{x-r_k}{r_{k+1}-r_k},$$
which  has no integral solution.
Since $t_{k+1}-r_{k+1}t_k/r_k=1/r_k\leq 1/2$, there is only one lattice point on the
vertical edge of $\Delta$, namely $(r_{k+1},t_{k+1})$.
\end{proof}

Set $P:=(l, lt_k/r_k)$, the intersection point of $\{x=l\}$ with $\{y=t_kx/r_k\}$;
 and let $Q$ (resp. $R$) be the intersection of $\{x=l\}$ with the segment $[(r_k,t_k),(r_{k+1},t_{k+1})]$
 (resp. with $y=\beta x/\alpha$).
 Let $M$ be the number of lattice points on the open  segment with ends $(r_k,t_k)$ and $P$. Then, by ($\dagger$) and  (\ref{lem:LP}),
\begin{equation}\label{M}
LP(\Delta^\circ_l)=LP(\Delta_{(0,0),(l,0),P}^\circ)+M+1.\end{equation}
(The last `1' counts $(r_k,t_k)$.)
Notice that $R$ is not a lattice point ($l\beta/\alpha\not\in \Z$),
on the line $x=l$ there is no lattice point strict between $R$ and  $Q$ (by $\dagger$),
$Q$ is not a lattice point and there is no lattice point in the interior of $[PQ]$ (by (\ref{lem:LP})). Hence
\begin{equation}\label{PP}
lt_k/r_k\geq \lfloor l\beta/\alpha\rfloor.
\end{equation}
Assume that in (\ref{PP}) one has equality, i.e.,
 $P$ is a lattice point. Then,  by Pick theorem,
$2LP(\Delta_{(0,0),(l,0),P}^\circ)=(l-1)(\lfloor l\beta/\alpha\rfloor-1)-1-M$.
This together with (\ref{M}) provides the needed strict inequality.

Finally, assume that $P$ is not a lattice point. Let $N$ be the number of lattice points on the interior of the segment with ends $(0,0)$ and $P':=(l,\lfloor l\beta/\alpha\rfloor)$. Then comparing the triangles
$\Delta_{(0,0),(l,0),P}$ and $\Delta_{(0,0),(l,0),P'}$, and by Pick theorem we get
$$2 LP(\Delta_l^\circ)\geq (l-1)(\lfloor
 l\beta/\alpha\rfloor-1)+N+M+1.
\qedhere$$
\end{proof}

Now we are ready to formulate the main result of this subsection.
For any $(\alpha,\beta)$ as above, consider the sequence $\{r_1,\ldots, r_\crr\}$ as in (\ref{prop:LP}) and
 define (following Wagreich and VanDyke (cf. \cite{wa}))
 $$f_{\alpha,\beta}(t)=\sum_{k=1}^\crr\ t^{r_k}.$$

The next theorem generalizes the result of VanDyke valid for singularities satisfying $b_0\geq \nu+1$,
\cite{vd}.

 \begin{theorem}\label{th:LP}
 Assume that $\Gamma$ has Seifert invariants $b_0$ and $\{(\alpha_i,\omega_i)\}_{i=1}^\nu$
with $b_0\geq \nu$. Set $\beta_i=\alpha_i-\omega_i$. Then
$$P_{\m_X/\m_X^2}(t)=(b_0-\nu+1)t+\sum_{i=1}^\nu\ f_{\alpha_i,\beta_i}(t).$$
 \end{theorem}
 \begin{proof}
 Clearly, $Q(1)=s_1+1=b_0-\nu+1$. For $l>1$, $Q(l)$ is given by (\ref{QL}) and (\ref{prop:LP}).
 \end{proof}

 \subsection{The case of automorphic forms.}
 Assume that $G$ is a Fuchsian group (of the first kind) with signature
 $(g=0; s; \alpha_1,\ldots,\alpha_\nu)$, where $s\geq 0$.
 Let $A(G)$ be the graded ring of automorphic forms relative to $G$; for details see \cite{wa2}.
 Then by \cite[(5.4.2)]{wa2}, $X=Spec(A(G))$ is a normal weighted homogeneous singularity whose star--shaped
 graph has Seifert invariants $b_0=\nu+s-2$, and $\{(\alpha_i,\omega_i)\}_{i=1}^\nu$ with all $\omega_i=1$.
The values $s\geq 2$, $s=1$ resp. $s=0$ correspond exactly to the fact that
$X$ is minimal rational, non--minimal rational or minimally elliptic, cf. (\cite[(5.5.1)]{wa2}).

 Wagreich in \cite[Theorem (3.3)]{wa2} provides an incomplete list for $P_{\m_X/\m_X^2}(t)$ for all these singularities. By our method not only we can reprove his formulae, but we also
 complete his result clarifying all the possible cases.

 Notice that if $\beta=1$ then  $f_{\alpha,\beta}(t)=f_{\alpha,1}(t)=\sum_{k=2}^\alpha\, t^k$.

\begin{theorem}\label{th:W}
Assume that the Seifert invariants of the graph $\Gamma$ satisfy $\omega_i=1$ for all $i$ and $s=b_0-\nu+2\geq 0$
and $\nu\geq 3$. Then $P_{\m_X/\m_X^2}(t)$ has the next forms:

\noindent (I) \ The cases considered by Wagreich \cite{wa2}:
$$P_{\m_X/\m_X^2}(t)=f(t)+\sum_{i=1}^\nu\ f_{\alpha_i,1}(t),$$
where $f(t)$ is given by the following list
$$
\begin{array}{ll}
s\geq 2 & f(t)=(b_0-\nu+1)t\\
s=1 & f(t)=-t^2+(\nu-2)t^3\\
s=0, \ \nu\geq 4, \ \sum_i\alpha_i\geq 11 & f(t)=-3t^2+(\nu-5)t^3\\
s=0, \ \nu=3,\ \alpha_i\geq 3 \ \mbox{for all $i$}, \ \sum_i\alpha_i\geq 12 &
f(t)=-3t^2-2t^3-t^4\\
s=0, \ \nu=3,\ \alpha_1=2, \ \alpha_2,\alpha_3\geq 4, \ \sum_i\alpha_i\geq 13&
f(t)=-3t^2-2t^3-t^4-t^5\\
s=0, \ \nu=3,\ \alpha_1=2, \ \alpha_2=3, \ \alpha_3\geq 9&
f(t)=-3t^2-2t^3-t^4-t^5-t^7
\end{array}$$

\noindent (II). In all the remaining cases  $s=0$ (and $X$ is a hypersurface with $\gamma=1$),
and $P_{\m_X/\m_X^2}(t)$ is:
$$
\begin{array}{ll}
\Sf=(1,(2,3,7),\ub 1) &  t^6+t^{14}+t^{21}, \\
\Sf=(1,(2,3,8),\ub 1) &  t^6+t^8+t^{15},\\
\Sf=(1,(2,4,5),\ub 1) &  t^4+t^{10}+t^{15},\\
\Sf=(1,(2,4,6),\ub 1) &  t^4+t^6+t^{11},\\
\Sf=(1,(2,5,5),\ub 1) &   t^4+t^5+t^{10},\\
\Sf=(1,(3,3,4),\ub 1)  &  t^3+t^8+t^{12},\\
\Sf=(1,(3,3,5),\ub 1) &   t^3+t^5+t^9,\\
\Sf=(1,(3,4,4),\ub 1) &   t^3+t^4+t^8,\\
\Sf=(2,(2,2,2,3),\ub 1) &  t^2+t^6+t^9,\\
\Sf=(2,(2,2,2,4),\ub 1) &  t^2+t^4+t^7,\\
\Sf=(2,(2,2,3,3),\ub 1) &  t^2+t^3+t^6,\\
\Sf=(3,(2,2,2,2,2),\ub 1) &  2 t^2+t^5.\\
\end{array}$$
\end{theorem}
\begin{proof} The proof is a case by case verification based on the results of (\ref{ss:GP}).
\end{proof}

\section{Examples when $P_{\m_X/\m_X^2}(t)$ is not topological}\label{s:NOT}
\subsection{} Our goal is to present an
equisingular family of weighted homogeneous
 surface singularities in
 which the  embedding dimension is not constant.
As a consequence, the embedding dimension of weighted
 homogeneous singularities, in general,   is
 not a topological invariant (i.e., cannot be determined by
 the Seifert invariants).

\begin{example}\label{ss:223377} Consider the Seifert invariants
$\Sf=(2,(2,2,3,3,7,7),\ub 1)$.

Suppose that the numbering of $E_i$'s satisfies
$(-b_1, \dots, -b_6)=\ba$, and  we will use similar  notations as in (\ref{ex:2345}).
Here is the list of the  fundamental invariants:
\begin{enumerate}
 \item $e=-1/21$, $\alpha=42$, $\co=2$, $\gamma=43$.
 \item $|H|=84$.
 \item The fundamental cycle is $Z=\{3, 3, 2, 2, 1, 1,6\}$,
       $p_a(Z)=7$.
 \item The canonical cycle is $Z_K=\{22, 22, 15, 15, 7, 7, 44\}$.
 \item The Hilbert series is
 \begin{multline*}
	 P_{G_X}(t)=\frac{1-2 t-4 t^2-3 t^3+2 t^5+2
       t^6+t^7+2 t^8+2 t^9+t^{10}}
{1+t-t^3-t^4-t^7-t^8+t^{10}+t^{11}}+P_{H^1}(t) \\
=1+t^6+t^{12}+t^{14}+t^{18}+t^{20}
+t^{21}+t^{24}+t^{26}+t^{27}+t^{28}+t^{30} \\
+t^{32}+t^{33}+t^{34}+t^{35}+t^{
   36}+t^{38}+t^{39}+t^{40} \\
+t^{41}+3 t^{42}+t^{44}+t^{45}+t^{46}+t^{47}+3
   t^{48}+t^{49}+t^{50}+O\left(t^{51}\right).
       \end{multline*}
 \item $p_g=24$ and \begin{multline*}
	P_{H^1}(t)=3
   t+t^2+t^3+t^4+t^5+t^7+t^8+t^9+t^{10} \\
+t^{11}+t^{13}+t^{15}+t^{16}+t^{17}+t^{19} \\
+t^{22}+t^{23}+t^{25}+t^{2
   9}+t^{31}+t^{37}+t^{43}.
       \end{multline*}
\item The degrees of $z_1,\ldots z_6$ are $(\frac{21}{2}, \frac{21}{2}, \frac{21}{3}, \frac{21}{3},
\frac{21}{7} ,\frac{21}{7})$.
\end{enumerate}
\noindent {\bf Claim.}
$s_l\ge 0$ for $l\ge 44$.
In particular,
$X_{l+44}\ni 1$ for $l\ge \alpha=42$.
\begin{proof}
From the expression of $P_{H^1}(t)$, we read  $s_1=-4$ and
 $s_l\ge -2$ otherwise. Then use
 $s_{l+\alpha}=s_l+\co$. For the second statement use
(\ref{r:trivial})(4).
\end{proof}
Then, by a computation (for $l\leq 86$) one verifies that the only values $l$ for which $1\not\in X_l$ are
$l\in\{6,14,21,42\}$. In the first three cases $s_l=0$ and $X_l=\emptyset$.  For $l=42$ one has
$s_{42}=2$ and $X_{42}=\{a_1a_2,a_3a_4,a_5a_6\}$. It depends essentially on the choice of the
full rank matrix $(a_{ji})$, cf. (\ref{NT}). In particular,
$$
e.d.(X,o)= \begin{cases}
			  3 & \mbox{if} \ p_1p_2-p_3p_4\ne0, \\ 4 & \mbox{if} \ p_1p_2-p_3p_4=0;
			 \end{cases}
$$where for the  notations $\{p_k\}$, see  (\ref{NT}).

Next, let us find a system of equations for $X$.
We take the following splice diagram equations:
$$
p_1x_1^2+x_2^2+x_3^3, \quad
p_2x_1^2+x_2^2+x_4^3, \quad
p_3x_1^2+x_2^2+x_5^7, \quad
p_4x_1^2+x_2^2+x_6^7.
$$
The action of the group $H$  is given by the following diagonal matrix:
$$
[-i,i,e^{\frac{ i \pi }{3}},e^{-\frac{ i \pi }{3}},e^{\frac{ i \pi }{7}},e^{-\frac{ i \pi }{7}}].
$$
Then $R^H$ is
generated by the following 21 monomials (in the next computations we used SINGULAR \cite{SING}):
\begin{gather*}
 z_{5}z_{6} , \; z_{3}z_{4} , \; z_{1}z_{2} , \; z_{2}^{4} ,
 \; z_{1}^{4} , \; z_{2}^{2}z_{4}^{3} , \;
 z_{1}^{2}z_{4}^{3} , \; z_{2}^{2}z_{3}^{3} , \;
 z_{1}^{2}z_{3}^{3} , \; z_{4}^{6} , \; z_{3}^{6} , \; \\
 z_{2}^{2}z_{6}^{7} , \;
z_{1}^{2}z_{6}^{7} , \; z_{2}^{2}z_{5}^{7} , \;
z_{1}^{2}z_{5}^{7} , \; z_{4}^{3}z_{6}^{7} , \;
z_{3}^{3}z_{6}^{7} , \; z_{4}^{3}z_{5}^{7} , \;
z_{3}^{3}z_{5}^{7} , \; z_{6}^{14} , \; z_{5}^{14}.
\end{gather*}
These are denoted by $y_1, \dots , y_{21}$. Then $X$ is
defined by the following ideal:
$$
\begin{array}{ll}
y_{9}+y_{3}^{2}+p_{1}y_{5}, & y_{11}-2p_{1}y_{3}^{2}-y_{4}-p_{1}^{2}y_{5}, \\
y_{8}+p_{1}y_{3}^{2}+y_{4}, & y_{10}-2p_{2}y_{3}^{2}-y_{4}-p_{2}^{2}y_{5}, \\
y_{7}+y_{3}^{2}+p_{2}y_{5}, &  y_{2}^{3}-(p_{1}+p_{2})y_{3}^{2}-y_{4}-p_{1}p_{2}y_{5}, \\
y_{6}+p_{2}y_{3}^{2}+y_{4}, & y_{19}-(p_{1}+p_{3})y_{3}^{2}-y_{4}-p_{1}p_{3}y_{5}, \\
y_{3}^{4}-y_{4}y_{5}, & y_{18}-(p_{2}+p_{3})y_{3}^{2}-y_{4}-p_{2}p_{3}y_{5}, \\
y_{15}+y_{3}^{2}+p_{3}y_{5}, & y_{17}-(p_{1}+p_{4})y_{3}^{2}-y_{4}-p_{1}p_{4}y_{5}, \\
y_{14}+p_{3}y_{3}^{2}+y_{4}, & y_{16}-(p_{2}+p_{4})y_{3}^{2}-y_{4}-p_{2}p_{4}y_{5}, \\
y_{13}+y_{3}^{2}+p_{4}y_{5}, & y_{21}-2p_{3}y_{3}^{2}-y_{4}-p_{3}^{2}y_{5}, \\
y_{12}+p_{4}y_{3}^{2}+y_{4}, & y_{20}-2p_{4}y_{3}^{2}-y_{4}-p_{4}^{2}y_{5}, \\
y_{1}^{7}-(p_{3}+p_{4})y_{3}^{2}-y_{4}-p_{3}p_{4}y_{5} &
\end{array}
$$
By eliminating the variables except for $y_1$, $y_2$, $y_3$,
and $y_5$,
we obtain
$$
\begin{array}{c}
y_{3}^{4}-y_{2}^{3}y_{5}+(p_{1}+p_{2})y_{3}^{2}y_{5}+p_{1}p_{2}y_{5}^{2}, \\
y_{1}^{7}-y_{2}^{3}+(p_{1}+p_{2}-p_{3}-p_{4})y_{3}^{2}+(p_{1}p_{2}-p_{3}p_{4})y_{5}
\end{array}
$$
Notice that if $p_{1}p_{2}-p_{3}p_{4}\not=0$, then $y_5$ can also be eliminated.
\end{example}

\begin{example} It is instructive to consider the graph with Seifert invariants
 $\Sf=(3,(2,2,3,3,7,7),\ub 1)$ too; it has the same legs as (\ref{ss:223377}),
 but its $b_0$ is $-3$ (instead of $-2$).
By this `move', we modify the graph into
 the `direction of rational graphs'. The consequence is that
 the ambiguity of the example (\ref{ss:223377}) disappears, and $P_{\m_X/\m_X^2}(t)$
 becomes topological (and the embedding dimension increases).

In this case
$e=-22/21$, $\alpha=42$, $\gamma=43/22$,  $P_{H^1}(t)=2t$. Since $\gamma <2$,
$s_l\ge -1$ for $l\ge 2$. Since $s_3=1$, $s_l\ge 0$ for
$l\ge 5$. Then $X_l\ni 1$ for $l\ge 5+\alpha=47$.

The following is the list of $(l, s_l, X_l)$ with $s_l\ge 0$
and $1\not \in X_l$ (and  where
$l=(1, 1, 0, 0, 0, 0)$, $m=(0, 0, 1, 1, 0, 0)$,
 $n=(0, 0, 0, 0, 1, 1)$).
$$
\begin{array}{ccc}
 2 & 0 & \emptyset  \\
 3 & 1 & \emptyset  \\
 4 & 2 & \{n\} \\
 5 & 3 & \{n\} \\
 6 & 6 & \{m+n,l+n\} \\
 7 & 5 & \{n\} \\
 14 & 14 & \{n,m+n,l+n,l+m\} \\
  21 & 21 & \{n,m,m+n\} \\
  42 & 44 & \{n,m,m+n,l,l+n,l+m,l+m+n\} \\
\end{array}
$$
By a computation one gets
$$P_{\m_X/\m_X^2}(t)=t^2+2t^3+2t^4+2t^5+2t^6+2t^7.$$
\end{example}

\section{Splice--quotients with star--shaped graphs}
\subsection{General discussion.} In this section we  extend
our study to the case of splice--quotient singularities.
We briefly recall their definition under the assumption that their graphs is
star--shaped.

Let $\Gamma$ be as above, and
let $\{f_j\}$ be the set of Brieskorn polynomials considered in
subsection (\ref{ss:neumann}), and  let $X$ be the quotient (weighted homogeneous singularity)
of their zero--set by the action of $H$.
Let $\{g_j(z_1, \dots ,z_{\nu})\}_{j=1}^{\nu-2}$ be a set of power series satisfying  the
following conditions:
\begin{itemize}
 \item  with respect to the weights
       $\mathrm{wt}(z_i)=(|e|\alpha_i)^{-1}$,
the degree of the leading form of $g_j$ is bigger
       than $\deg f_j$ for any $j$;
 \item all monomials in $g_j$ are  elements of
       $\mu$-eigenspace $R^{\mu}$ (recall that $\{f_j\}
       \subset R^{\mu}$).
\end{itemize}

Then the singularity $Y:=\defset{(z_i)\in
\C^{\nu-2}}{f_j+g_j=0, \; j=1, \dots, \nu-2}/H$ is a normal
surface singularity; it is called a
splice--quotient (see \cite{nw-CIuac}).
All these singularities belong to an equisingular family sharing
the same resolution graph $\Gamma$.

By  upper-semicontinuity one gets
\begin{prop}\label{prop:semicon}
 $e.d.(Y,o) \le e.d.(X,o)$.
\end{prop}

Here is the main question of the present
section: is $P_{\m_X/\m_X^2}(t)$  topological (i.e. independent of the choice of
the matrix $(a_{ji})$ and of the power series $g_{j}$) 
--- at least when for the weighted homogeneous case
with the corresponding graph the answer is positive ?
Here the graded module structure of $\m_X/\m_X^2$ is induced
by the natural weighted filtration of the local ring
$\cO_{X,o}$ defined by the weights  $\mathrm{wt}(z_i)=(|e|\alpha_i)^{-1}$.

The answer splits in two parts. First, if the graph is rational, 
or Gorenstein 
 elliptic, then by a semi--continuity argument, the 
discussion of subsection (\ref{ss:NR}) remain true for splice 
quotients as well. Moreover, the combinatorial formulas of 
theorems (\ref{th:LP}) (valid for $b_0\geq \nu$)
and  of (\ref{th:W}) (valid for $\omega_i=1$ and $b_0-\nu\geq -2$) 
are true for splice quotients too
(as far as the corresponding combinatorial assumptions on the graph are satisfied).

On the other hand, we do not expect that the positive results proved for weighted homogeneous singularities listed in (\ref{ss:O1}) (the case $\co=1$),
in (\ref{ss:OS}) ($\co$ `small'), or in (\ref{cor:nusmall}) ($\nu\leq 5$) will remain valid for splice quotients too.
A counterexample in the case $\co=1$ is analyzed in the next subsection.
This example also emphasizes that in the inequality (\ref{prop:semicon}) 
 the strict inequality might occur.

\subsection{}
In \cite[(4.5)]{supiso}, it is shown that the embedding dimension is not constant
in an equisingular deformation of a weighted homogeneous
singularity.
However, in that example, the general fibre is not a
splice--quotient.
The aim of this section is to show the following:

\begin{prop}\label{p:nontop}
There exists an equisingular  deformation of a weighted homogeneous
singularity  in
 which the  embedding dimension is not constant, and it
 satisfies the following:
\begin{enumerate}
 \item The embedding dimension of the central fibre is
       determined by the Seifert invariants (i.e. all the weighted homogeneous singularities with
       the same link have the same embedding  dimension).
 \item Every fibre is a splice-quotient.
\end{enumerate}
\end{prop}

We will analyze with more details the  case
 $\co=1$, i.e. when $\{f_j\} \subset R^H$ and $\deg (f_j)=\alpha$.
In this case the ideals $I_X, I_Y
\subset R^H$ of $X$ and $Y$ are generated by $\{f_j\}$, $\{f_j+g_j\}$,
respectively.
Let $l_X$ (resp. $l_Y$) be the rank of the image of
$\{f_j\}_j$ (resp. $\{f_j+g_j\}_j$) in $\m/\m^2$.
Then $e.d.(X,o)=\dim (\m/\m^2)-l_X$ (and similarly for $Y$).
Hence,  $e.d.(X,o)-e.d.(Y,o)=l_Y-l_X$.

If $l$ denotes the number of linear monomials of degree
$>\alpha$,
then by letting $g_j$ be general linear combinations of those
monomials, we obtain $l_Y=\min(\nu-2, l_X+l)$, which in special situations can be larger than
$l_X$. Taking special linear combinations, we get for $l_Y$:
$$l_X\leq l_Y\leq \min(\nu-2, l_X+l)$$
and any value between the two combinatorial  bounds can be realized.
The following example shows the existence of a  deformation $Y$
which verifies (\ref{p:nontop}).

\begin{ex}\label{LINBIG2}
 Let $X$ be the weighted homogeneous singularity considered in
(\ref{LINBIG}). Recall that there are 9 linear monomials in $R^H$, $l_X=2$ and the
 embedding dimension is topological, it is 7.
Moreover,  $\alpha=420$, and there are two linear monomials of degree greater than $420$:
$\deg (z_1^2z_4^2z_5^2)=460$, $\deg(z_1z_4^4z_5)=440$. Let us take
 the following  equations for $Y$ (with $c,d \in \C$):
\begin{align*}
 &z_1^3+z_3^5 + z_5^{21}=0\\
 &z_4^6 +z_3^5 + z_5^{21}=0\\
 &z_2^4 +z_3^5 + z_5^{21}+c z_1z_4^4z_5 +d z_1^2z_4^2z_5^2=0.
\end{align*}
Then $l_Y=2$ if $c=d=0$, but  $l_Y=3$ (hence $e.d.(Y,o)=6$) otherwise.
\end{ex}

\end{document}